\documentclass[leqno]{article}

\usepackage[frenchb,english]{babel}
\usepackage[T1]{fontenc}

\usepackage{amsmath}
\usepackage{amssymb}
\usepackage{amsfonts}
\usepackage{enumerate}
\usepackage{vmargin}
\usepackage[all]{xy}
\usepackage{mathrsfs}
\usepackage{mathtools}
\usepackage{lmodern}
\usepackage[colorlinks=true,pagebackref=true]{hyperref}
\usepackage{comment}
\setmarginsrb{3cm}{3cm}{3.5cm}{3cm}{0cm}{0cm}{1.5cm}{3cm}

\newcommand{\A}{\ensuremath{\mathcal{A}}}
\newcommand{\E}{\ensuremath{\mathbb{E}}}

\newcommand{\B}{\mathrm{B}} 
\newcommand{\I}{\mathrm{I}} 
\newcommand{\II}{\mathrm{II}} 
\newcommand{\K}{\ensuremath{\mathbb{K}}}
\let\L\relax 
\newcommand{\L}{\mathrm{L}}
\let\O\relax
\newcommand{\O}{\ensuremath{\mathbb{O}}}

\newcommand{\M}{\mathrm{M}}

\newcommand{\R}{\ensuremath{\mathbb{R}}}

\newcommand{\C}{\ensuremath{\mathbb{C}}}

\newcommand{\W}{\mathrm{W}}
\newcommand{\kR}{\mathfrak{R}}
\newcommand{\Zc}{\mathrm{Z}}
\newcommand{\can}{\mathrm{can}}

\newcommand{\Id}{\mathrm{Id}}

\newcommand{\JW}{\mathrm{JW}}
\newcommand{\JB}{\mathrm{JB}}
\newcommand{\JBW}{\mathrm{JBW}}

\newcommand{\la}{\langle}\newcommand{\ra}{\rangle}
\renewcommand{\leq}{\ensuremath{\leqslant}}
\renewcommand{\geq}{\ensuremath{\geqslant}}

\newcommand{\qed}{\hfill \vrule height6pt  width6pt depth0pt}
\newcommand{\norm}[1]{ \| #1  \|}
\newcommand{\bnorm}[1]{ \big\| #1  \big\|}

\newcommand{\co}{\colon}

\newcommand{\ot}{\otimes}
\newcommand{\ovl}{\overline}

\newcommand{\sa}{\mathrm{sa}}

\let\i\relax 
\newcommand{\i}{\mathrm{i}}
\newcommand{\ov}{\overset}

\newcommand{\epsi}{\varepsilon}
\renewcommand{\d}{\mathop{}\mathopen{}\mathrm{d}} 
\newcommand{\e}{\mathrm{e}} 
\renewcommand{\d}{\mathop{}\mathopen{}\mathrm{d}} 
\let\cal\relax
\newcommand{\cal}{\mathcal}

\DeclareMathOperator{\tr}{Tr} 
\DeclareMathOperator{\Ran}{Ran} 
\DeclareMathOperator{\sgn}{sgn} 
\let\Re\relax 
\DeclareMathOperator{\Re}{Re} 
\newtheorem{thm}{Theorem}[section]
\newtheorem{defi}[thm]{Definition}
\newtheorem{prop}[thm]{Proposition}
\newtheorem{conj}[thm]{Conjecture}
\newtheorem{quest}[thm]{Question}

\newtheorem{lemma}[thm]{Lemma}

\newtheorem{remark}[thm]{Remark}

\newtheorem{example}[thm]{Example}

\newenvironment{proof}[1][]{\noindent {\it Proof #1} : }{\hbox{~}\qed
\smallskip
}

\usepackage{tocloft}
\numberwithin{equation}{section}
\usepackage[nottoc,notlot,notlof]{tocbibind}

\let\OLDthebibliography\thebibliography
\renewcommand\thebibliography[1]{
  \OLDthebibliography{#1}
  \setlength{\parskip}{0pt}
  \setlength{\itemsep}{0pt plus 0.3ex}
}

\begin{document}
\selectlanguage{english}
\title{\bfseries{Nonassociative $\L^p$-spaces and embeddings in noncommutative $\L^p$-spaces}}
\date{}
\author{\bfseries{C\'edric Arhancet}}

\maketitle

\begin{abstract}
We define a notion of nonassociative $\L^p$-space associated to a $\JBW^*$-algebra (Jordan von Neumann algebra) equipped with a normal faithful state $\varphi$. In the particular case of $\JW^*$-algebras underlying von Neumann algebras, we connect these spaces to a complex interpolation theorem of Ricard and Xu on noncommutative $\L^p$-spaces. We also make the link with the nonassociative $\L^p$-spaces of Iochum associated to $\JBW$-algebras and the investigation of contractively complemented subspaces of noncommutative $\L^p$-spaces. More precisely, we show that our nonassociative $\L^p$-spaces contain isometrically the $\L^p$-spaces of Iochum and that all tracial nonassociative $\L^p$-spaces from $\JW^*$-factors arise as positively contractively complemented subspaces of noncommutative $\L^p$-spaces. 
\end{abstract}

\makeatletter
 \renewcommand{\@makefntext}[1]{#1}
 \makeatother
 \footnotetext{
\noindent {\it 2020 Mathematics subject classification:}
Primary 46L51, 46B70, 17C65.

{\it Key words and phrases}: noncommutative $\L^p$-spaces, nonassociative $\L^p$-spaces, projections, complemented subspaces, $\JBW$-algebras, interpolation.}

{
  \hypersetup{linkcolor=blue}
\tableofcontents
}

\section{Introduction}

The investigation of the structure of projections and complemented subspaces is a classical topic of Banach space geometry. If $1 \leq p < \infty$, a classical result of Ando \cite{And66} says that the contractively complemented subspaces of a classical $\L^p$-space $\L^p(\Omega)$ for a finite measure space $\Omega$ are all isometrically isomorphic to an $\L^p$-space. Moreover,  it is known \cite[Problem 5.4.1]{AbA02} \cite[Theorem 4.10]{Ran01} that a subspace $Y$ of an $\L^p$-space is the range of a positive contractive projection if and only if $Y$ is order isometric to some $\L^p$-space. In the important works \cite{ArF78} and \cite{ArF92}, a complete description of contractively complemented subspaces of the Schatten class $S^p$ was given by Arazy and Friedman for any $1 \leq p \leq \infty$. Such a subspace is isometrically isomorphic to a $\ell^p$-sum of $S^p$-Cartan factors of type I-IV. Note that these Cartan factors are rectangular operators spaces, spaces of antisymmetric operators, spaces of symmetric operators and complex spin factors. 

Using a new and different approach than the one of \cite{ArF92}, it is achieved in \cite{ArR19} that the range of a 2-positive contractive projection on an arbitrary noncommutative $\L^p$-space is completely order and completely isometrically isomorphic to some noncommutative $\L^p$-space. Furthermore, it is showed that a 2-positive contractive projection (in the $\sigma$-finite case) can be seen as a conditional expectation modulo a symmetric two-sided change of density. By deepening the method, we obtained in the paper \cite{Arh23} a description of positively contractively complemented subspaces of noncommutative $\L^p$-spaces associated to a $\sigma$-finite von Neumann algebra and we have highlighted the crucial role of Jordan conditional expectations on von Neumann algebras in this topic, whose range is a $\JW^*$-algebra. More precisely, we showed that a positively contractively complemented subspace of a noncommutative $\L^p$-space is isometric to the range of an $\L^p$-extension of a suitable Jordan conditional expectation. Moreover, the interpolation couples of the type $(\cal{M},\cal{M}_*)_{\frac{1}{p}}$, where $\cal{M}$ is a $\JW^*$-algebra, appear in this context.

Recall that a $\JW^*$-algebra is a weak* closed Jordan-$*$-subalgebra of a von Neumann algebra. This notion was essentially introduced by Edwards in the paper \cite{Edw80}. Actually, he introduced the larger class of $\JBW^*$-algebras, including the exceptional $\JBW^*$-factor $\mathrm{H}_3(\O_\C)$ of all Hermitian ($3 \times 3$)-matrices over the algebra $\O_\C$ of complex octonions. It is worth noting that a von Neumann algebra $\cal{M}$ equipped with the Jordan product 
\begin{equation}
\label{Jordan-product}
x \circ y 
\ov{\mathrm{def}}{=} \frac{1}{2}(xy+yx), \quad x,y \in \cal{M} 
\end{equation}
is an example of $\JW^*$-algebra, hence a $\JBW^*$-algebra. It is known that the selfadjoint parts of $\JBW^*$-algebras and $\JW^*$-algebras are precisely the $\JBW$-algebras and the $\JW$-algebras. The last ones are the  real linear spaces of selfadjoint operators which are closed for the weak* topology and closed under the Jordan product $\circ$. In the continuity of the classical work \cite{JNW34} of Jordan, von Neumann and Wigner on Jordan algebras, the theory of $\JW$-algebras was introduced by Topping in \cite{Top65} and developed by several authors. Initially, these Jordan algebras were defined for understanding the notion of observable in quantum mechanics. We refer to the books \cite{AlS03}, \cite{ARU97}, \cite{CGRP14}, \cite{CGRP18}, \cite{HOS84}, to the important paper \cite{Sto66} and references therein for more information on these algebras. Note that by \cite[Theorem 3.9]{Shu79} 
a $\JBW$-algebra $A$ can be uniquely decomposed as a direct sum 
\begin{equation}
\label{decompo}
A
=A_{\mathrm{sp}} \oplus A_{\textrm{exp}}
\end{equation}
where $A_{\mathrm{sp}}$ is a $\JW$-algebra and $A_{\textrm{exp}}$ is a purely exceptional $\JBW$-algebra, i.e.~isomorphic to the algebra $\mathrm{C}(X,\mathrm{H}_3(\O))$ of continuous functions from $X$ with values in the $\JBW$-factor $\mathrm{H}_3(\O)$ for a hyperstonean compact Hausdorff space $X$. Recall that $\O$ stands for the algebra of octonions.

These are related to the study of the state spaces of $\mathrm{C}^*$-algebras, see \cite{AHOS80}, \cite{AlS78} and the survey \cite{Alf79}. Furthermore, these algebras are connected to bounded symmetric domains, see \cite[pp.~92-93]{HOS84} for a brief overview of this topic. Finally, the category of $\sigma$-finite $\JBW$-algebras is equivalent to the category of facially homogeneous self-dual cone in real Hilbert spaces, see \cite{Ioc84}. This fact can be seen as a generalization of the well-known one-to-one correspondence between $\sigma$-finite von Neumann algebras and orientable facially homogeneous self-dual cones in complex Hilbert spaces introduced by Connes in \cite{Con74}. We also refer to \cite{AlS98} for an  explanation of the relevance of the concept of orientation for the passage from Jordan structures to associative structures in operator algebras.

In this paper, we study interpolation spaces $\L^p(\cal{M},\varphi) \ov{\mathrm{def}}{=} (\cal{M},\cal{M}_*)_{\frac{1}{p}}$ defined by an interpolation couple $(\cal{M},\cal{M}_*)$, where $\cal{M}$ is a $\JBW^*$-algebra and $\cal{M}_*$ is its predual and where the compatibility is obtained with the help of a normal faithful state $\varphi$ on $\cal{M}$. We will show that these spaces can be seen as nonassociative $\L^p$-spaces. In particular, we will prove the duality relation $(\L^p(\cal{M},\varphi))^*=\L^{p^*}(\cal{M},\varphi)$ and, if $\varphi$ is a \textit{trace}, a version of H\"older's inequality. In the particular case of the underlying $\JW^*$-algebra of a von Neumann algebra equipped with a normal \textit{tracial} state, we recover Kosaki's noncommutative $\L^p$-spaces defined in \cite{Kos84}. Finally, in the paper \cite{Arh23b}, we will generalize the nonassociative $\L^p$-spaces of this paper by introducing the $\L^p$-spaces $\L^p(X,\varphi)$ of a ($\sigma$-finite) $\JBW^*$-triple $X$ equipped with a \textit{suitable} functional $\varphi \in X_*$ satisfying $\norm{\varphi}_{X_*}=1$. Recall that a $\JBW^*$-algebra $\cal{M}$ admits a canonical structure of $\JBW^*$-triple.

Our construction allows us to interpret in Example \ref{Ex-AVN} a complex interpolation result of Ricard and Xu \cite[Corollary 1.2]{RiX11}, which has its roots in the remarkable paper \cite{JuP08}. Indeed, it gives an isomorphic description of the norms of the nonassociative $\L^p$-spaces associated with the canonical underlying $\JW^*$-algebra of a von Neumann algebra. Furthermore, we connect our spaces to the nonassociative $\L^p$-spaces $\L^{p,I}(A,\tau)$ introduced by Iochum in \cite{Ioc86}, which are \textit{real} Banach spaces associated to a $\JBW$-algebra $A$ equipped with a normal finite faithful trace $\tau$. We will show that our associative $\L^p$-space $\L^p(\cal{M},\varphi)$ contains isometrically the space $\L^{p,I}(A,\tau)$ if $\varphi$ is the complexification of the trace $\tau$, see Theorem \ref{Th-Iochum}.

Our last result is the following theorem which says that a large class of \textit{tracial} nonassociative $\L^p$-spaces arise as positively contractively complemented subspaces of noncommutative $\L^p$-spaces. It can be seen as a kind of converse to the main result of \cite{Arh23b}.

\begin{thm}
\label{th-embedding-intro}
Let $\cal{M}$ be a $\JW^*$-factor with separable predual equipped with a normal finite faithful trace $\tau$. Suppose that $1 \leq p < \infty$. Then the Banach space $\L^p(\cal{M},\tau)$ is isometric to a positively 1-complemented subspace of a noncommutative $\L^p$-space associated with a finite von Neumann algebra.
\end{thm}

\paragraph{Structure of the paper}
The paper is organized as follows. In Section \ref{Sec-Jordan}, we recall some information on Jordan algebras. In Section \ref{Sec-Lp-state-JBWstar}, we introduce nonassociative $\L^p$-spaces associated to $\JBW^*$-algebras by using complex interpolation. Moreover, we give an interpretation of the complex interpolation result of Ricard and Xu in Example \ref{Ex-AVN}. Furthermore, we describe the precise link with the nonassociative $\L^p$-spaces of Iochum in Theorem \ref{Th-Iochum}. In Section \ref{Sec-nonassociative-Lp-spaces}, we show that a \textit{tracial} nonassociative $\L^p$-space of a $\JW^*$-factor is isometric to a positively contractively complemented subspace of a noncommutative $\L^p$-space. Finally, we raise in Section \ref{sec-open-questions} several open problems related to the content of this paper.

\section{Preliminaries}
\label{Sec-Jordan}

\paragraph{Jordan algebras} A Jordan algebra $A$ over a field $\K$ is a vector space $A$ over $\K$ equipped with a (not necessarily associative) commutative bilinear product $\circ$ that satisfies $x \circ (x^2 \circ y) =x^2\circ (x \circ y)$ for any $x,y \in A$, see e.g.~\cite[Definition 1.1 p.~3]{AlS03} or \cite[p.~162]{CGRP14}. This means that the multiplication operators by $x$ and $x^2$ commute.


Following \cite[Definition 1.5 p.~5]{AlS03} and \cite[3.1.4 p.~76]{HOS84}, a $\JB$-algebra is a Jordan algebra over the scalar field $\R$ equipped with a complete norm satisfying the properties 
\begin{equation}
\label{def-JB}
\norm{x \circ y} \leq \norm{x} \norm{y}, \quad \norm{x^2}=\norm{x}^2 
\quad \text{and} \quad \norm{x^2} \leq \norm{x^2+y^2}
\end{equation}
for any $x,y \in A$. A $\JBW$-algebra is a $\JB$-algebra which is a dual Banach space \cite[p.~111]{HOS84}. In this case, the predual is unique.

A $\JB^*$-algebra \cite[p.~91]{HOS84} \cite[Definition 3.3.1 p.~345]{CGRP14} is a complex Banach space $\cal{M}$ which is a complex Jordan algebra equipped with an involution satisfying 
\begin{equation}
\label{def-JBstar}
\norm{x \circ y} \leq \norm{x} \norm{y},\quad  \norm{x^*}=\norm{x} 
\quad \text{and} \quad \norm{\{x,x^*,x\}}=\norm{x}^3
\end{equation}
for any $x,y \in \cal{M}$, where we use the Jordan triple product $\{x,y,z\}\ov{\mathrm{def}}{=} (x \circ y) \circ z+(y \circ z) \circ x-(x \circ z) \circ y$. An element $x$ in a $\JB^*$-algebra is called selfadjoint if $x^* = x$. The selfadjoint part of a $\JB^*$-algebra is the real subalgebra consisting of all selfadjoint elements. A $\JBW^*$-algebra \cite[p.~4]{CGRP18} is a $\JB^*$-algebra which is a dual Banach space. The selfadjoint part of a $\JBW^*$-algebra has a canonical structure of $\JBW$-algebra by \cite[Corollary 5.1.29 p.~9]{CGRP18}. Conversely, if $A$ is a $\JBW$-algebra then by \cite[Corollary 5.1.41 p.~15]{CGRP18} there exists a unique $\JBW^*$-algebra $\cal{M}$ such that $A$ is the selfadjoint part of $\cal{M}$. We also refer to \cite[pp.~21-22]{BHK17} and \cite[Corollary 5.1.29 p.~9]{CGRP18} for useful results which can be used for transferring results from $\JBW$-algebras to $\JBW^*$-algebras and vice versa. In particular, if $\varphi \in \cal{M}_*$ satisfies $\varphi(x)=\ovl{\varphi(x^*)}$ for any $x \in \cal{M}$, we have
\begin{equation}
\label{norm-VU}
\norm{\varphi}_{\cal{M}_*}
=\norm{\varphi|A}_{A_*}.
\end{equation}
The notion of positivity in a $\JB$-algebra and in a $\JB^*$-algebra is defined in \cite[Section 3.3]{HOS84} and \cite[p.~9]{CGRP18}.

%


\paragraph{Centers and factors} Two elements $a$ and $b$ of a Jordan algebra $A$ are said to operator commute \cite[p.~44]{HOS84} if for any $z \in A$ we have $(x \circ z) \circ y= x \circ (z \circ y)$. The centre $\Zc(A)$ of $A$ is the set of all elements of $A$ which operator commute with all elements of $A$. An element $x \in A$ is central if it belongs to the center. By \cite[Lemma 2.5.3 p.~45]{HOS84}, the centre is an associative subalgebra of $A$. Following \cite[p.~115]{HOS84}, if the centre of a $\JBW$-algebra $A$ only consists of scalar multiples of the identity, we say that $A$ is a $\JBW$-factor. We refer to \cite[Theorem 6.1.40 p.~362]{CGRP18} for a classification of $\JBW$-factors.

If $p$ is a projection (i.e.~$p \circ p=p$) of a $\JBW$-algebra, the smallest central projection $q$ such that $q \geq p$ is called the central cover of $p$ and denoted by $c(p)$ \cite[Definition 2.38 p.~56]{AlS03}. We say that a projection $p$ of a $\JBW$-algebra $A$ is abelian if the algebra $A_p \ov{\mathrm{def}}{=} \{\{p,x,p\} : x \in A\}$ is associative \cite[p.~122]{HOS84}. 


\paragraph{$\JW$-algebras} Recall that a (concrete) $\JW$-algebra \cite[p.~95]{HOS84} \cite[Definition 2.70 p.~70]{AlS03} 
 is a weak* closed Jordan subalgebra of $\B(H)_\sa$, that is a real linear space of selfadjoint operators which is closed for the weak* topology and closed under the Jordan product $\circ$. Here $H$ is a complex Hilbert space. Note that a $\JW$-algebra is a $\JBW$-algebra by \cite[p.~95]{HOS84}. By \cite[Proposition 1.49 p.~28]{AlS03}, two elements $x$ and $y$ of a $\JW$-algebra operator commute if and only if $x$ and $y$ commute in $\B(H)$.


A $\JW$-algebra $A$ is said to be reversible \cite[Definition 4.24]{AlS03} \cite[p.~25]{HOS84} if it is closed under symmetric products, i.e.~if $a_1,\ldots,a_k \in A$ then 
$$
a_1a_2 \cdots a_k+a_k\cdots a_2a_1 \in A.
$$

\paragraph{$\JW^*$-algebras} We define a $\JW^*$-algebra as a complex subspace of $\B(H)$ which is closed for the weak* topology and closed under the Jordan product $\circ$ and the involution, for some complex Hilbert space $H$.  
A $\JW^*$-algebra is a $\JBW^*$-algebra. The selfadjoint part of a $\JW^*$-algebra is a $\JW$-algebra. Conversely, if $A$ is a $\JW$-algebra included in $\B(H)$ then the complexification $A_\C=A+\i A$ is a $\JW^*$-algebra included in $\B(H)$. 



\begin{example} \normalfont
\label{spin-factor}
Let $\cal{H}$ be a separable real Hilbert space of dimension at least 2, and let $\R 1$ denote a one dimensional real Hilbert space with unit vector 1. Let $A \ov{\mathrm{def}}{=} \cal{H} \oplus \R 1$ and consider the product $\circ$ on $A$ defined by
\begin{equation}
\label{product-spin-factor}
(a+\lambda 1)  \circ  (b  + \mu 1)  
\ov{\mathrm{def}}{=} \mu a  + \lambda b  +  (\langle a ,  b\rangle  +  \lambda \mu)1, \quad a,b \in \cal{H}, \lambda,\mu \in \R
\end{equation}
and the norm $\norm{a+\lambda 1}_A  \ov{\mathrm{def}}{=} \norm{a}_\cal{H}  + |\lambda|$. Then $A$ is a $\JBW$-factor by \cite[Proposition 3.37 p.~92]{AlS03}, called (real) spin factor which admits a concrete representation as a $\JW$-algebra by \cite[Theorem 4.1 p.~103]{AlS03}, which can be reversible or not. These algebras were introduced by Topping in \cite{Top65} and \cite{Top66}. He showed in \cite[Theorem 3]{Top66} (see also \cite[Proposition 6.1.5 p.~137]{HOS84}) that two spin factors are isomorphic if and only if their real Hilbert space dimensions are equal. 

Now, we recall the standard matrix representations of spin factors. Let
\[
\sigma_1
\ov{\mathrm{def}}{=} \begin{bmatrix}
   1  & 0  \\
   0  & -1  \\
\end{bmatrix}
, \quad
\sigma_2 \ov{\mathrm{def}}{=} \begin{bmatrix} 
0&1\\ 
1&0
\end{bmatrix}
, \quad
\sigma_3
\ov{\mathrm{def}}{=} \begin{bmatrix} 
0&\i\\ 
-\i&0
\end{bmatrix}
\]
be the Pauli matrices. Let $n \geq 1$ be a integer. Define the matrices
$s_1 \ov{\mathrm{def}}{=} \sigma_1 \ot \I_2^{\ot n-1},  
s_2 \ov{\mathrm{def}}{=} \sigma_2 \ot \I_2^{\ot n-1},s_3 \ov{\mathrm{def}}{=} \sigma_3 \ot \sigma_1 \ot \I_2^{\ot n-2}, 
s_4 \ov{\mathrm{def}}{=} 
\sigma_3 \ot \sigma_2\ot \I_2^{\ot n-2},\ldots,
s_{2i-1} \ov{\mathrm{def}}{=} \sigma_3^{\ot i-1} \ot \sigma_1 \ot \I_2^{\ot n-i}, 
s_{2i} \ov{\mathrm{def}}{=} 
\sigma_3^{\ot i-1} \ot \sigma_2 \ot \I_2^{\ot n-i},
\ldots,
s_{2n-1} \ov{\mathrm{def}}{=} \sigma_3^{\ot n-1} \ot \sigma_1,
s_{2n} \ov{\mathrm{def}}{=} \sigma_3^{\ot n-1 } \ot \sigma_2$ of $\M_{2^n}(\mathbb{C})$.  Note that it is known that $\{s_1,\ldots,s_k\}$ is a spin system, i.e.~consists of symmetries $\not=\pm \Id$ such that $s_is_j=-s_is_j$ if $i \not= j$.

If $k \in \{2n-1,2n\}$ with $k \geq 2$, the real linear span of the set $\{1,s_1,\ldots,s_k\}$ is a $\JW$-factor which is isomorphic to a spin factor by \cite[pp.~140-141]{HOS84} and \cite[pp.~103-104]{AlS03}. More precisely, if $(e_k)$ is an orthonormal basis of the real Hilbert space $\cal{H}$ with $\dim \cal{H} =k$, then $A \to \{1,s_1,\ldots,s_k\}$, $e_k \mapsto s_k$ defines an isomorphism.

%

For the infinite dimensional spin factor, we need a bigger algebra. To this end, we consider for any integer $n \geq 1$, the normal injective $*$-homomorphism $i_n \co \M_{2^n}(\mathbb{C}) \to \M_{2^{n+1}}(\mathbb{C})$ given by
\begin{equation}
\label{TowerEmbed}
i_n(x) 
\ov{\mathrm{def}}{=} x \ot \I_2=\begin{bmatrix}
x & 0 \\ 
0 & x
\end{bmatrix},\quad x \in \M_{2^n}(\mathbb{C}).
\end{equation}
The family $(\M_{2^n}(\mathbb{C}),i_n)_{n \geq 1}$ is a directed system of $\mathrm{C}^*$-algebras in the sense of \cite[p.~864]{KaR97b}. The inductive limit of this system is a $\mathrm{C}^*$-algebra $\cal{A}$ called CAR algebra (<<canonical anticommutation relations>>) and more precisely an infinite tensor product, which contains each $\M_{2^n}(\mathbb{C})$. We can see each $s_i$ as an element of $\cal{A}$. The closed real linear span of $\{1,s_1,s_2\ldots\}$ is an infinite-dimensional spin factor contained in $\cal{A}$.
\end{example}

\paragraph{Type $\I_n$} Let $n$ be some cardinal number. Following \cite[p.~86]{AlS03}, we say that a $\JBW$-algebra $A$ is of type $\I_n$ if there exists a family $(p_i)$ of $n$ abelian projections in $A$ such that $1=\sum_i p_i$ and $c(p_i)=1$ for any $i$. 

\begin{example} \normalfont
\label{ex-type-II-JW}
The $\JBW$-algebras of type $\I_2$ were classified by Stacey in \cite[Theorem 2]{Sta82}. If $A$ is a $\JBW$-algebra with separable predual then $A$ has type $\I_2$ if and only if there exist an index set $I$, a family $(\Omega_i)_{i \in I}$ of second countable locally compact spaces, a family $(\mu_i)_{i \in I}$ of Radon measures on the spaces $\Omega_i$ and a family $(S_i)_{i \in I}$ of spin factors, each of dimension strictly greater than 1 and at most countable giving an isomorphism 
$$
A
=\oplus_{i \in I} \L^\infty_\R(\Omega_i,S_i).
$$ 
\end{example}

%





\paragraph{Traces} A trace on a $\JBW$-algebra $A$ is a function $\tau$ defined on the subset $A_+$ of positive
elements of $A$ with values in $[0,+\infty]$ satisfying the following conditions:
\begin{enumerate}
\item $\tau(x+y)=\tau(x)+\tau(y)$ for any $x, y \in A_+$,
\item $\tau(\lambda x)=\lambda\tau(x)$ for any $x \in A_+$ and any $\lambda \geq 0$, where $0.(+\infty)=0$,
\item $\tau(s \circ x \circ s)=\tau(x)$ for any $x \in A_+$ and any arbitrary symmetry $s$ of $A$.
\end{enumerate}
Recall that a symmetry is an element $s$ such that $s \circ s=1$. The trace $\tau$ is said to be faithful if $\tau(x) > 0$ for all non-zero $x \in A_+$, finite if $\tau(1) < + \infty$, semifinite if given any non-zero $x \in A_+$ there is a non-zero $y \in A_+$ such that $y \leq x$ with $\tau(y) < +\infty$. The trace $\tau$ is normal if for every increasing net $(x_i)$ of positive elements such that $x_i \to x$ where $x \in A_+$, we have $\tau(x_i) \to \tau(x)$. We refer to \cite{AyA85}, \cite{Ayu82}, \cite{Ayu92}, \cite{Kin83} and \cite{PeS82} for more information on traces on $\JBW$-algebras.

Every finite trace on a $\JBW$-algebra $A$ can be extended by linearity to a linear functional on $A$. Thus a finite trace on a $\JBW$-algebra $A$ can be seen as a positive linear functional $\tau$ satisfying the condition $\tau(s \circ x \circ s)=\tau(x)$ for all $x \in A$ and all symmetries $s \in A$. By \cite[Lemma 5.18 p.~147]{AlS03}, it is known that the last condition is equivalent to the formula 
\begin{equation}
\label{Def-trace}
\tau (x \circ (y \circ z)) 
= \tau((x \circ y) \circ z), \quad x,y,z \in A.
\end{equation}
By complexification with \cite[Corollary 5.1.41 p.~15]{CGRP18}, we obtain a normal faithful positive linear functional on the associated $\JBW^*$-algebra $\cal{M}$ satisfying 
\begin{equation}
\label{trace}
\tau (x \circ(y \circ z))
=\tau((x \circ y) \circ z), \quad x,y,z \in \cal{M}.
\end{equation}
We say that such a map is a normal finite faithful trace on $\cal{M}$.








\begin{example} \normalfont
\label{ex-trace-spin} Consider a spin factor $S=\cal{H} \oplus \R 1$ as in Example \ref{spin-factor}. By \cite[Lemma 5.21 p.~149]{AlS03} (see also \cite[Proposition 6.1.7 p.~137]{HOS84} and \cite{Top66}), there exists a unique tracial state $\tau$. Moreover, the same reference shows that $\tau$ is normal, faithful and defined by $\tau(1)=1$ and $\tau(a)=0$ for any $a \in \cal{H}$. 
\end{example}


\paragraph{Jordan conditional expectations} We say that a positive map $T \co A \to A$ on a $\JB^*$-algebra $A$ is faithful if $T(x)=0$ for some $x \in A_+$ implies $x=0$. 

Let $\cal{N}$ be a unital $\JW^*$-subalgebra of a $\JW^*$-algebra $\cal{M}$. A linear map $Q \co \cal{M} \to \cal{M}$ is said to be a Jordan conditional expectation on $\cal{N}$ if it is a unital positive map of range $\cal{N}$ which is $\cal{N}$-modular, that is 
\begin{equation}
\label{Def-cond-exp-JBstar}
Q(x \circ Q(y))
=Q(x) \circ Q(y), \quad x,y \in \cal{M}.
\end{equation}
With $x=1$ and $y \in \cal{N}$, we obtain $Q(y)=Q(1 \circ y) \ov{\eqref{Def-cond-exp-JBstar}}{=} Q(1) \circ y=1 \circ y=y$. It follows that $Q$ is the identity on $\cal{N}$. Consequently, $Q$ is a projection.

\paragraph{Selfadjoint maps} Let $\cal{M}$ be a von Neumann algebra equipped with a normal semifinite faithful trace $\tau$. Recall that a \textit{positive} normal contraction $T \co \cal{M} \to \cal{M}$ is selfadjoint with respect to $\tau$ \cite[p.~49]{JMX06} if for any $x,y \in \cal{M} \cap \L^1(\cal{M})$ we have $\tau(T(x)y)=\tau(xT(y))$. We have a similar notion for a normal faithful state $\varphi$ on $\cal{M}$ instead of the trace $\tau$, see \cite[p.~122]{JMX06}. 
We have the following observation of \cite{Arh23} which will be used in Section \ref{Sec-nonassociative-Lp-spaces}.

\begin{prop}
\label{Prop-selfadjoint}
Let $\cal{M}$ be a von Neumann algebra equipped with a normal semifinite faithful trace. Let $Q \co \cal{M} \to \cal{M}$ be a trace preserving normal Jordan conditional expectation. Then $Q$ is selfadjoint.   
\end{prop}


%

\paragraph{Existence of Jordan conditional expectations}
Let $A$ be a $\JBW$-algebra and $B$ a $\JBW$-subalgebra of $A$. Suppose that $\tau$ is a normal faithful tracial state on $A$. If we also denote by $\tau$ the restriction of $\tau$ on $B$, it is essentially showed in \cite[Theorem 4.2]{HaS95} (combined with \cite[Remark 3.7]{HaH84}) that there exists a faithful normal Jordan conditional expectation $Q \co A \to A$ onto $B$ such that $\tau \circ Q=\tau$. See \cite[Theorem p.~78]{Edw86} for a previous preliminary result without proof. We give a complete elementary proof (i.e.~without modular theory) of this result in Section \ref{Existence}. By complexification, we obtain the following result.

\begin{prop}
\label{prop-Jordan-conditional-expectation-existence-bis}
Let $\cal{M}$ be a $\JBW^*$-algebra and $\cal{N}$ be a $\JBW^*$-subalgebra of $\cal{M}$. Let $\tau$ be a normalized normal finite faithful trace on $\cal{M}$. Then there exists a trace preserving normal faithful Jordan conditional expectation $Q \co \cal{M} \to \cal{M}$ on $\cal{N}$.
\end{prop}

\paragraph{Interpolation} 
We start by recalling some background on complex interpolation theory. We refer to the books \cite{BeL76},  \cite{KPS82} and \cite{Lun18} for more information. Let $X_0,X_1$ be two Banach spaces which embed into a topological vector space space $\tilde{X}$. We say that $(X_0,X_1)$ is an interpolation couple. Then the sum $X_0 + X_1 \ov{\mathrm{def}}{=} \{x \in \tilde{X} : x=x_1+x_2 \text{ for some } x_1 \in X_1,x_2 \in X_2\}$ is well-defined and equipped with the norm
$$
\norm{x}_{X_0+X_1}
\ov{\mathrm{def}}{=} \inf_{x=x_0+x_1} \big(\norm{x_0}_{X_0}+\norm{x_1}_{X_1}\big).
$$
The intersection $X_0 \cap X_1$ is equipped with the norm 
$$
\norm{x}_{X_0 \cap X_1}
\ov{\mathrm{def}}{=} \max\{\norm{x}_{X_0},\norm{x}_{X_1}\}.
$$

Consider the closed strip $\ovl{S} \ov{\mathrm{def}}{=} \{z \in \mathbb{C} : 0 \leq \Re z \leq 1\}$. Let us denote by $\mathscr{F}(X_0,X_1)$ the family of bounded continuous functions $f \co \ovl{S} \to X_0+X_1$, holomorphic on the open strip $S \ov{\mathrm{def}}{=} \{z \in \mathbb{C} : 0< \Re z <1\}$ inducing continuous functions $\R\to X_0$, $t \mapsto f(\i t)$ and $\R \to X_1$, $t \mapsto f(1+\i t)$ which tend to 0 when $|t|$ goes to $\infty$. For any $f \in \mathscr{F}(X_0,X_1)$, we set
\begin{equation}
\label{norm-funct-inter}
\norm{f}_{\mathscr{F}(X_0,X_1)}
\ov{\mathrm{def}}{=} \max \left\{\sup_{t \in \R} \norm{f(\i t)}_{X_0}, \sup_{t \in \R} \norm{f(1+\i t)}_{X_1}\right\}.
\end{equation}
If $0 \leq \theta \leq 1$, we define the subspace $
(X_0,X_1)_\theta
\ov{\mathrm{def}}{=} \big\{f(\theta) : f \in \mathscr{F}(X_0,X_1)\big\}$ of the Banach space $X_0+X_1$. For any $x \in (X_0,X_1)_\theta$, we let
$$
\norm{x}_{(X_0,X_1)_\theta}
\ov{\mathrm{def}}{=} \inf \big\{\norm{f}_{\mathscr{F}(X_0,X_1)} : f \in \mathscr{F}(X_0,X_1), f(\theta)=x\big\}.
$$
Then by \cite[Theorem 4.1.2]{BeL76} $(X_0,X_1)_\theta$ equipped with this norm is a Banach space. 

In the particular case $X_0 \subset X_1$, the sum $X_0+X_1$ is isometric to $X_1$ and the intersection $X_0 \cap X_1$ is isometric to the space $X_0$. By \cite[Proposition 2.4 p.~50]{Lun18}, we have the contractive inclusions
\begin{equation}
\label{contract-inclusions}
X_0 \subset (X_0,X_1)_\theta \subset X_1.
\end{equation}
Moreover, by \cite[Theorem 4.2.2 p.~91]{BeL76} the subspace $X_0$ is dense in the Banach space $(X_0,X_1)_\theta$.

Recall the classical following result \cite[Theorem 4.4.1 p.~96]{BeL76} on bilinear interpolation.

\begin{thm}
\label{Th-bilinear-interpolation}
Let $(X_0,X_1)$, $(Y_0,Y_1)$ and $(Z_0,Z_1)$ be an interpolation couple of Banach spaces. Assume that $T \co X_0 \cap X_1\times Y_0 \cap Y_1 \to Z_0 \cap Z_1$ is a bilinear map satisfying  
$$
\norm{T(x,y)}_{Z_0} \leq M_0 \norm{x}_{X_0} \norm{y}_{Y_0}
$$
and
$$
\norm{T(x,y)}_{Z_1} \leq M_1 \norm{x}_{X_1} \norm{y}_{Y_1}
$$
for any $x \in X_0 \cap X_1$ and any $y \in Y_0 \cap Y_1$. Then $T$ induces a bounded bilinear map from $(X_0,X_1)_\theta \times (Y_0,Y_1)_\theta$ into the space $(Z_0,Z_1)_\theta$ with norm at most $M_0^{1-\theta} M_1^\theta$.
\end{thm}

We will use the following result which is the particular case of \cite[Theorem 3.3]{Wer18} when $X_0 \subset X_1$. See the papers \cite{CoS98}, \cite{HaP89}, \cite{LiP64}, \cite{Wat95} and \cite{Wat00}   
for variants of this result.

\begin{thm}
\label{Th-Werner}
Let $(X_0,X_1)$ be a couple of Banach spaces with $X_0 \subset X_1$ such that $X_1^* = X_0$ (with duality bracket $\la \cdot, \cdot \ra_{X_0,X_1}$) satisfying the following conditions.
\begin{enumerate}
	\item $X_0$ is dense in $X_1$.
	\item For any $x \in X_0$, we have 
$
\norm{x}_{X_1} 
=\sup \big\{|\la y, x \ra_{X_0,X_1} | : y \in X_0, \norm{y}_{X_0} \leq 1 \big\}
$.
	\item For any linear functional $\psi$ on $X_0$ which is continuous with respect to the norms of $X_0$ and $X_1$ there exists some $z \in X_0$ such that
$$
\psi(x) 
=\la z,x \ra_{X_0,X_1}, \quad x \in X_0.
$$
\item On $X_0$ there exists an involution $*$ such that $
\la x, y \ra_{\cal{H}} 
\ov{\mathrm{def}}{=} \la y, x^*\ra_{X_1,X_0}$ is a scalar product on $X_0$.

\item Let $\cal{H}$ denote the completion of $X_0$ with respect to this scalar product. Assume that $*$ is isometric with respect to the norms of $X_0$, $X_1$ and $\cal{H}$.
\end{enumerate}
Then $(X_0,X_1)_{\frac{1}{2}}
=\cal{H}$ (the norms on $X_0$ coincide). Moreover, for any $0 < \theta < 1$ we have 
$$
(X_0,X_1)_\theta^*
=(X_0,X_1)_{1-\theta}.
$$
\end{thm}


We finish with the following result \cite[Lemma 3.1]{Wer18}.

\begin{prop}
\label{Prop-Werner}
Let $(X_0,X_1)$ be an interpolation couple of Banach spaces whose intersection $X_0 \cap X_1$ is dense in $X_0$ and in $X_1$. Let $T \co X_0 \cap X_1 \to X_0 \cap X_1$ be an anti-linear surjective map that is isometric in both norms $\norm{\cdot}_{X_0}$ and $\norm{\cdot}_{X_1}$. Then for any $0 < \theta < 1$, the map $T$ induces an anti-linear isometry on the space $(X_0,X_1)_{\theta}$.
\end{prop}

\paragraph{Projections} We will use the next well-known result \cite[Theorem 1 p.~118]{Tri95}, which will help us to determine the range of some projections constructed by interpolation. 

\begin{lemma}
\label{Lemma-interpolation}
Let $(X_0,X_1)$ be an interpolation couple and let $C$ be a contractively complemented subspace of $X_0+X_1$. We assume that the corresponding contractive projection $P \co X_0+X_1 \to X_0+X_1$ satisfies $P(X_i) \subset X_i$ and that the restriction $P \co X_i \to X_i$ is contractive for $i=0,1$. Then $(X_0 \cap C,X_1 \cap C)$ is an interpolation couple and the canonical inclusion $J \co C \to X_0+X_1$ induces an isometric isomorphism $\tilde{J}$ from $(X_0 \cap C,X_1 \cap C)_\theta$ onto the subspace $P((X_0,X_1)_\theta)=(X_0,X_1)_\theta \cap C$ of the space $(X_0,X_1)_\theta$. 
\end{lemma}

The following is \cite[Theorem 3.2.6 p.~297]{Meg98} combined with \cite[5.10 p.~148]{FHHMPZ01}.

\begin{prop}
\label{Prop-Fabian}
Let $X$ be a Banach space and consider a bounded map $P \co X \to X$. Then $P$ is a projection if and only if $P^*\co X^* \to X^*$ is a projection. In this case, the space $P(X)^*$ is isomorphic to the space $P^*(X^*)$.
\end{prop}

\paragraph{Complex analysis} We recall the Hadamard three-lines theorem which we will use. We refer to \cite[p.~33]{ReS72} for a proof.

\begin{prop}
\label{prop-three-lined}
Let $f \co \ovl{S} \to \mathbb{C}$ be a bounded and continuous function, holomorphic on the open strip $S$. For any $\theta \in [0,1]$ we let $M_\theta \ov{\mathrm{def}}{=} \sup_{y \in \R} |f(\theta+\i y)|$. Then
$$
M_\theta 
\leq M_0^{1-\theta} M_1^\theta, \quad \theta \in [0,1].
$$
\end{prop}

\section{Nonassociative $\L^p$-spaces associated with a $\JBW^*$-algebra}
\label{Sec-Lp-state-JBWstar}

Let $\cal{M}$ be a $\JBW^*$-algebra with product $(x,y) \mapsto x \circ y$. We start by constructing an embedding of $\cal{M}$ in its predual $\cal{M}_*$ with the help of a normal faithful state $\varphi$ in order to have a compatible couple $(\cal{M},\cal{M}_*)$ in the sense of the complex interpolation theory.

Let $\varphi$ be an element of the predual $\cal{M}_*$. For any $x \in \cal{M}$, we define the functional $\varphi_x \co \cal{M} \to \mathbb{C}$ by
\begin{equation}
\label{Translate-state}
\varphi_x(y)
\ov{\mathrm{def}}{=} \varphi(x \circ y), \quad y \in \cal{M}.
\end{equation}

\begin{lemma}
\label{Lemma-contractivity}
The map $\varphi_x \co \cal{M} \to \C$ is weak* continuous and satisfies $\norm{\varphi_x}\leq \norm{\varphi}\norm{x}_{\cal{M}}$.
\end{lemma}

\begin{proof}
Recall that the product of a $\JBW^*$-algebra is separately weak* continuous, see e. g. \cite[Corollary 5.1.30 (iii) p.~11]{CGRP18}. Hence the map $y \mapsto x \circ y$ is weak* continuous. So by composition the functional $\varphi_x$ is also weak* continuous. In a $\JBW^*$-algebra, the product $\circ$ is contractive by definition. Consequently, for any $y \in \cal{M}$, we have
$$
|\varphi_x(y)|
\ov{\eqref{Translate-state}}{=} |\varphi(x \circ y)| \leq \norm{\varphi}\norm{x \circ y}_{\cal{M}} 
\ov{\eqref{def-JBstar}}{\leq} \norm{\varphi} \norm{x}_{\cal{M}} \norm{y}_{\cal{M}}.
$$
\end{proof}

\begin{prop}
\label{Prop-injection}
Let $\cal{M}$ be a $\JBW^*$-algebra equipped with a normal faithful state $\varphi$. The linear map $i \co \cal{M} \to \cal{M}_*$, $x \mapsto \varphi_x$ is contractive and injective. Moreover, its range is dense in the Banach space $\cal{M}_*$.
\end{prop}

\begin{proof}
The first assertion is a consequence of Lemma \ref{Lemma-contractivity}. Let $x \in \cal{M}$. Suppose $\varphi_x=0$. We have
$$
\varphi(x \circ x^*)
\ov{\eqref{Translate-state}}{=} \varphi_x(x^*)
=0.
$$ 
By \cite[(5.1.2) p.~9]{CGRP18} \cite[Lemma 3.1]{HKPP20}, we have $x \circ x^* \geq 0$. By the faithfulness of $\varphi$, we infer that $x \circ x^*=0$. By \cite[Lemma 3.1]{HKPP20} \cite[Lemma 3.4.65 p.~382]{CGRP14}, we deduce $x=0$. We conclude that the map $i$ is injective.

If $y \in \cal{M}$ belongs the annihilator $i(\cal{M})^\perp$ of $i(\cal{M})$, we have 
$$
\varphi(y^* \circ y)
\ov{\eqref{Translate-state}}{=} \varphi_{y^*}(y)
=\la i(y^*), y\ra_{\cal{M}_*,\cal{M}}
=0.
$$ 
Using again the faithfulness of $\varphi$, we obtain $y^* \circ y=0$. We infer that $y=0$. We conclude that $i(\cal{M})^\perp=\{0\}$. By \cite[Proposition 1.10.15 (c) p.~93]{Meg98}, the range of $i$ is dense in the Banach space $\cal{M}_*$. 
\end{proof}
 
Hence $(\cal{M},\cal{M}_*)$ is an interpolation couple of complex Banach spaces. So we can use the complex interpolation method for defining nonassociative $\L^p$-spaces in the spirit of the construction of noncommutative $\L^p$-spaces of \cite{Kos84}. The sum $\cal{M}+\cal{M}_*$ is isometric to the predual $\cal{M}_*$ and the intersection  $\cal{M} \cap \cal{M}_*=i(\cal{M})$ is isometric to $\cal{M}$. 

\begin{defi}
\label{def-nonasso-Lp}
Let $\cal{M}$ be a $\JBW^*$-algebra equipped with a normal faithful state $\varphi$. Suppose that $1 < p < \infty$. We let
\begin{equation}
\label{Def-Lp-state-JBW}
\L^p(\cal{M},\varphi)
\ov{\mathrm{def}}{=} (\cal{M},\cal{M}_*)_{\frac{1}{p}}.
\end{equation}
We say that this space is the nonassociative $\L^p$-space associated with $\cal{M}$ and $\varphi$.
\end{defi}
We shall write simply $\L^p(\cal{M})$ when there is no ambiguity on the functional $\varphi$. We let $\L^\infty(\cal{M}) \ov{\mathrm{def}}{=} \cal{M}$ and $\L^1(\cal{M}) \ov{\mathrm{def}}{=} \cal{M}_*$. By \eqref{contract-inclusions}, we have the contractive inclusions $\cal{M}\subset \L^p(\cal{M}) \subset \cal{M}_*$. Moreover, by \cite[Theorem 4.2.2 p.~91]{BeL76} the subspace $\cal{M}$ is dense in the Banach space $\L^p(\cal{M})$.

\begin{remark} \normalfont
We just need a faithful normal positive functional instead of a normal faithful state.
\end{remark}

Recall that an element $p$ of a $\JBW^*$-algebra is said to be a projection if $p^* = p$ and $p \circ p = p$ and that projections $p$ and $q$ are called orthogonal if $p \circ q = 0$. A $\JBW^*$-algebra is $\sigma$-finite if any orthogonal family of non-zero projections is countable. We easily provide the next folklore characterization of $\sigma$-finiteness.

\begin{prop}
A $\JBW^*$-algebra is $\sigma$-finite if and only if it admits a normal faithful state.
\end{prop}

\begin{proof}
An implication is \cite[Lemma 7.3]{HKPP20}. Conversely, suppose that $\JBW^*$-algebra $\cal{M}$ admits a normal faithful state $\varphi$. Given any orthogonal family $(p_i)_{i \in I}$ of non-zero projections in $\cal{M}$, we have $\varphi(p_i)>0$ for any $i \in I$ and 
$$
\sum_{i \in I} \varphi(p_i)
=\varphi\bigg(\sum_{i \in I} p_i\bigg) 
<\infty.
$$
Consequently, the index set $I$ is countable. Thus $\cal{M}$ is $\sigma$-finite.
\end{proof}


\begin{example} \normalfont
\label{Ex-AVN}
As we said in the introduction, a von Neumann algebra $\cal{M}$ equipped with the Jordan product $\circ$ defined in \eqref{Jordan-product} is a $\JW^*$-algebra (hence a $\JBW^*$-algebra). Let $\varphi$ be a normal faithful state on $\cal{M}$ and consider the density operator $D_{\varphi}$ associated with $\varphi$ in the context of the Haagerup noncommutative $\L^p$-spaces $\L^{p,H}(\cal{M})$ introduced in \cite{Haa79}. Recall that by \cite[(1.13)]{HJX10}, $D_\varphi$ belongs to the Banach space $\L^{1,H}(\cal{M})$ and that we can recover the state $\varphi$ by 
\begin{equation}
\label{HJX-1.13}
\varphi(x)
=\tr(D_\varphi x)
=\tr(xD_\varphi), \quad x \in \cal{M}.
\end{equation} 
For any $x,y \in \cal{M}$, we have
\begin{align}
\MoveEqLeft
\la i(x),y \ra_{\cal{M}_*,\cal{M}}
=\varphi_x(y)
\ov{\eqref{Translate-state}}{=} \varphi(x \circ y)
\ov{\eqref{Jordan-product}}{=} \varphi\bigg(\frac{1}{2}(xy+yx)\bigg) \label{inter-45} \\
&=\frac{1}{2}\big[\varphi(xy)+\varphi(yx)\big]
=\frac{1}{2}(x\varphi +\varphi x)(y) \nonumber
\end{align}
where $x\varphi \ov{\mathrm{def}}{=} \varphi(\cdot\, x)$ and $\varphi x \ov{\mathrm{def}}{=} \varphi(x\,\cdot)$. In this case, the linear map $i \co \cal{M} \to \cal{M}_*$, $x \mapsto \varphi_x$ is given by $x \mapsto \frac{1}{2}(x\varphi +\varphi x)$. Note, that this embedding is different from the ones of \cite{Kos84}.

Recall that $\Phi \co \L^{1,H}(\cal{M}) \to \cal{M}_*$, $z \mapsto \tr(z\, \cdot)$ is an isometric isomorphism. For any $x,y \in \cal{M}$, note that
\begin{align*}
\MoveEqLeft
\la i(x),y \ra_{\cal{M}_*,\cal{M}}
\ov{\eqref{inter-45}}{=} \frac{1}{2}(x\varphi +\varphi x)(y)
=\frac{1}{2}\big[ \varphi(y x)+\varphi(x y)\big] 
\ov{\eqref{HJX-1.13}}{=} \frac{1}{2}\big[ \tr(yxD_{\varphi} ) + \tr(D_{\varphi} xy)\big] \\
&=\tr\bigg(\frac{1}{2}(yxD_{\varphi} + D_{\varphi} xy)\bigg) 
=\bigg\la \Phi\bigg(\frac{1}{2}(xD_{\varphi}+D_{\varphi} x)\bigg),y \bigg\ra_{\cal{M}_*,\cal{M}}.
\end{align*}
So the embedding $i \co \cal{M} \to \cal{M}_*$ identifies to the embedding $\iota_1 \co \cal{M} \to \L^{1,H}(\cal{M})$, $x \mapsto \frac{1}{2}(xD_{\varphi} +D_{\varphi} x)$ considered in the paper \cite[p.~2]{RiX11} (up to the constant $\frac{1}{2}$). By \cite[Corollary 1.2]{RiX11}, we deduce a norm equivalence
\begin{equation*}
\label{}
\bnorm{D_{\varphi}^{\frac{1}{p}}x+xD_{\varphi}^{\frac{1}{p}}}_{\L^{p,H}(\cal{M})} 
\lesssim_p \norm{x}_{\L^p(\cal{M},\varphi)}
\lesssim_p \bnorm{D_{\varphi}^{\frac{1}{p}}x+xD_{\varphi}^{\frac{1}{p}}}_{\L^{p,H}(\cal{M})},\quad x \in \cal{M}.
\end{equation*}
Here the notation $\lesssim_p$ means an inequality $\leq K_p$ for a constant $K_p \geq 0$ only depending on $p$.

Note that in the particular case where the state $\varphi$ is a \textit{trace}, the linear map $i \co \cal{M} \to \cal{M}_*$, $x \mapsto \varphi_x$ is given by $x \mapsto \frac{1}{2}(x\varphi +\varphi x)=x\varphi$. This embedding is used in the paper \cite{Kos84}. So the space $\L^p(\cal{M},\varphi)$ is isometric to the Dixmier noncommutative $\L^p$-space $\L^{p,D}(\cal{M})$ introduced in \cite{Dix53}, since Kosaki noncommutative $\L^p$-spaces are isometric to Dixmier noncommutative $\L^p$-spaces.
\end{example}

Now, we will see that we can equip a nonassociative $\L^p$-space with a canonical anti-linear involution. Let $C \co \cal{M} \to \cal{M}$, $x \mapsto x^*$ be the conjugation map. If $x \in \cal{M}$ and $\psi \in \cal{M}_*$, we let $
\psi^*(x) 
\ov{\mathrm{def}}{=} \ovl{\psi(x^*)}$. Since the involution $*$ is weak* continuous by \cite[Theorem 5.1.29 (ii) p.~9]{CGRP18}, note that $\psi^*$ belongs to $\cal{M}_*$. Clearly, we have $(\psi^{*})^*=\psi$. 

\begin{prop}
\label{Prop-involution}
Let $\cal{M}$ be a $\JBW^*$-algebra equipped with a normal faithful state $\varphi$. Suppose that $1 \leq p \leq \infty$. The restriction of the map $C_* \co \cal{M}_* \to \cal{M}_*$, $\psi \to \psi^*$ on the subspace $\L^p(\cal{M})$ induces an isometric anti-linear map $C_p \co \L^p(\cal{M}) \to \L^p(\cal{M})$, $h \mapsto h^*$. Moreover, we have $C_\infty=C$.
\end{prop}

\begin{proof}
For any $x,y \in \cal{M}$, we have
$$
(\varphi_x)^*(y)
=\ovl{\varphi_x(y^*)}
\ov{\eqref{Translate-state}}{=} \ovl{\varphi(x \circ y^*)}
=\varphi(x^* \circ y)
=\varphi_{x^*}(y).
$$
We deduce that $(\varphi_x)^* = \varphi_{x^*}$ for any $x \in \cal{M}$. So we have the following commutative diagram.
$$
\xymatrix @R=1cm @C=2cm{
    \cal{M}_* \ar[r]^{C_*}   & \cal{M}_*    \\
   \cal{M} \ar[u]^{i} \ar[r]_{C} & \cal{M} \ar[u]_{i} 
}.
$$
For any $x \in \cal{M}$, note that 
$$
|\psi^*(x)|
=\big|\ovl{\psi(x^*)}\big|
=|\psi(x^*)| 
\leq \norm{\psi} \norm{x^*} 
\ov{\eqref{def-JBstar}}{=} \norm{\psi} \norm{x}.
$$ 
So $\norm{\psi^*} \leq \norm{\psi}$. Since $(\psi^{*})^*=\psi$, we have $\norm{\psi^*}=\norm{\psi}$. We conclude that $C_*$ is an isometry. By Proposition \ref{Prop-Werner}, we deduce an anti-linear isometric map $C_p \co \L^p(\cal{M}) \to \L^p(\cal{M})$. 
\end{proof}

Let $\varphi$ be a normal state on a $\JBW^*$-algebra $\cal{M}$. By the proof of \cite[Lemma 5.10.2 p.~275]{CGRP18} and \cite[Lemma 5.10.23 p.~286]{CGRP18}, the map $(x,y) \mapsto \varphi(x^* \circ y)$ is a positive definite hermitian sesquilinear form on the space $\cal{M}$. So we can consider introduce the complex scalar product
\begin{equation}
\label{Def-scalar-varphi}
\langle x, y \rangle_{\cal{H}_\varphi}
\ov{\mathrm{def}}{=}\varphi(x^* \circ y), \quad x,y \in \cal{M}
\end{equation}
with associated norm
\begin{equation}
\label{norm-H-varphi}
\norm{x}_{\cal{H}_\varphi}
=\sqrt{\varphi(x^* \circ x)}, \quad x \in \cal{M}.
\end{equation}
We denote by $\cal{H}_\varphi$ the associated Hilbert space which is the completion of $ \cal{M}$.

\begin{prop}
\label{L2-nonassociative}
Let $\cal{M}$ be a $\JBW^*$-algebra equipped with a normal faithful state $\varphi$. 
\begin{enumerate}
	\item The space $\L^2(\cal{M})$ is linearly isometric to the Hilbert space $\cal{H}_\varphi$. More precisely the norms of $\L^2(\cal{M})$ and of $\cal{H}_\varphi$ coincide on their common linear subspace $\cal{M}$.
	\item Suppose that $1 < p < \infty$. We have $(\L^p(\cal{M}))^*=\L^{p^*}(\cal{M})$ isometrically.
\end{enumerate}
\end{prop}

\begin{proof}
It suffices to check the assumptions of Theorem \ref{Th-Werner}. The first point is a consequence of Lemma \ref{Prop-injection}. The second point is obvious since for any $x \in \cal{M}$, we have 
$$
\norm{i(x)}_{\cal{M}_*} 
=\sup \big\{|\la i(x), y \ra_{\cal{M}_*,\cal{M}} | : y \in \cal{M}, \norm{y}_{\cal{M}} \leq 1 \big\}.
$$ 
For the third point, consider a linear functional $\psi$ on $ \cal{M}$ which is continuous with respect to the norms of $\cal{M}$ and $\cal{M}_*$. Since the subspace $i(\cal{M})$ is dense in $\cal{M}_*$, we can extend $\psi$ as a linear continuous form $\tilde{\psi}$ on $\cal{M}_*$. As $(\cal{M}_*)^*=\cal{M}$ there exists some $z \in \cal{M}$ such that
\begin{equation}
\label{der}
\tilde{\psi}(y) 
=\la y,z \ra_{\cal{M}_*,\cal{M}}, \quad y \in \cal{M}_*.
\end{equation}
We deduce that 
$$
\psi(x)
=\tilde{\psi}(i(x)) 
\ov{\eqref{der}}{=} \la i(x),z \ra_{\cal{M}_*,\cal{M}}, \quad x \in \cal{M}.
$$
The fourth point is true by using the scalar product \eqref{Def-scalar-varphi}. For any $x \in \cal{M}$, we have $\norm{x^*}_{\cal{H}_\varphi}
\ov{\eqref{norm-H-varphi}}{=} \sqrt{\varphi(x \circ x^*)}=\sqrt{\varphi(x^* \circ x)}\ov{\eqref{norm-H-varphi}}{=}\norm{x}_{\cal{H}_\varphi}$. Combined with Proposition \ref{Prop-involution}, this proves the last point.
\end{proof}


\begin{remark} \normalfont
Let $\cal{M}$ be a finite von Neumann algebra equipped with a normal faithful tracial state $\tau$. Then by essentially \cite[Remark p.~149]{AlS03}, the restriction of $\tau$ on the $\JW$-algebra $\cal{M}_{\sa}$ is a normal finite faithful trace. We have seen in the end of Example \ref{Ex-AVN}, we recover with the Banach space $\L^p(\cal{M},\tau)$ of \eqref{Def-Lp-state-JBW} the Dixmier noncommutative $\L^p$-space $\L^{p,D}(\cal{M})$. If $p=2$, we can see this fact by the following computation. For any $x \in \cal{M}$, we have 
$$
\norm{x}_{\cal{H}_\varphi}
\ov{\eqref{norm-H-varphi}}{=} \big(\tau (x^* \circ x)\big)^{\frac{1}{2}}
\ov{\eqref{Jordan-product}}{=} \bigg(\tau \bigg(\frac{1}{2}(x^* x+xx^*)\bigg)\bigg)^{\frac{1}{2}}
=\big(\tau (x^* x)\big)^{\frac{1}{2}}
=\norm{x}_{\L^{2,D}(\cal{M})}
$$
\end{remark}

\begin{prop}
Let $\cal{M}$ be a $\JBW^*$-algebra equipped with a normal faithful state $\varphi$. Suppose that $1 < p <\infty$. The Banach space $\L^p(\cal{M})$ is uniformly convex (hence reflexive) and uniformly smooth.
\end{prop}

\begin{proof}
Using the reiteration theorem \cite[Theorem 4.6.1 p.~101]{BeL76}, the Banach space $\L^p(\cal{M})$, $1 < p < 2$ is a complex interpolation space between the predual $\cal{M}_*$ and the Hilbert space $\L^2(\cal{M})$. By Proposition \ref{L2-nonassociative}, it follows from \cite{CwR82} that $\L^p(\cal{M})$ is uniformly convex. We can use a similar argument for $2 < p < \infty$, since in this case the space $\L^p(\cal{M})$ is a complex interpolation space between $\cal{M}$ and $\L^2(\cal{M})$. The uniform smoothness is obtained with the classical result \cite[Theorem 5.5.12 p.~500]{Meg98} which states that a normed space is uniformly convex if and only if its dual space is uniformly smooth.
\end{proof}

Now, we show that each $*$-automorphism of $\cal{M}$ and each normal contractive projection $\E \co \cal{M} \to \cal{M}$ onto a sub-$\JBW^*$-algebra, which preserve the state, induce canonical maps on the associated nonassociative $\L^p$-spaces. Recall that we have a contractive dense inclusion $\cal{M} \subset \L^p(\cal{M})$.

\begin{prop}
\label{prop-proj-Lp}
Let $\cal{M}$ be a $\JBW^*$-algebra equipped with a normal faithful state $\varphi$. Let $\alpha$ be a $*$-automorphism of $\cal{M}$ satisfying $\varphi \circ \alpha = \varphi$. Then, for each $1 \leq p < \infty$, the map $\alpha$ induces a surjective isometry $\alpha_p \co \L^p(\cal{M}) \to \L^p(\cal{M})$. Also, let $\E \co \cal{M} \to \cal{M}$ be a normal contractive projection from $\cal{M}$ onto a sub-$\JBW^*$-algebra $\cal{N}$ satisfying $\varphi \circ \E = \varphi$. Then, for each $1 \leq p < \infty$, $\E$ induces a projection $\E_p \co \L^p(\cal{M},\varphi) \to \L^p(\cal{M},\varphi)$ onto a subspace isometric to the nonassociative $\L^p$-space $\L^p(\cal{N},\varphi|\cal{N})$.
\end{prop}

\begin{proof}
Note that $\alpha \co \cal{M} \to \cal{M}$ is a surjective isometry by \cite[Proposition 3.4.4 p.~360]{CGRP14}. Due to the invariance $\varphi \circ \alpha = \varphi$, we have for any $x,y \in \cal{M}$
$$
\varphi_{\alpha(x)}(\alpha(y))
\ov{\eqref{Translate-state}}{=} \varphi(\alpha(x) \circ \alpha(y))
=\varphi(\alpha(x \circ y))
=\varphi(x \circ y)
\ov{\eqref{Translate-state}}{=} \varphi_x(y).
$$ 
So the map $i(\cal{M}) \to i(\cal{M})$, $\varphi_x \mapsto \varphi_{\alpha(x)}$, induces a surjective isometry $\alpha_1 \co \L^1(\cal{M}) \to \L^1(\cal{M})$, which sends $i(\cal{M}) = \cal{M}$ into itself isometrically. Thus the result follows by interpolation since a linear contraction with contractive linear inverse is necessarily isometric. The second assertion can be proved by similar arguments with Lemma \ref{Lemma-interpolation}.
\end{proof}

\begin{prop}
\label{prop-module}
Let $\cal{M}$ be a $\JBW^*$-algebra equipped with a normal faithful tracial state $\varphi$. Suppose that $1 \leq p \leq \infty$. The product $(x,y) \mapsto x \circ y$ on $\cal{M}$ induces contractive maps $\cal{M} \times \L^p(\cal{M}) \to \L^p(\cal{M})$, $(h,k) \mapsto h \circ k$ and $\L^p(\cal{M}) \times \cal{M} \to \L^p(\cal{M})$, $(h,k) \mapsto h \circ k$. 
\end{prop}

\begin{proof}
Consider the bilinear maps defined by
\begin{equation}
\label{Def-product-bis-bis}
P_1 \co i(\cal{M}) \times \cal{M}_* \to \cal{M}_*, (\varphi_x,\psi) \to \psi_x
\quad \text{and} \quad
P_0 \co i(\cal{M}) \times i(\cal{M}) \to i(\cal{M}), (\varphi_x,\varphi_y) \to \varphi_{x \circ y}.
\end{equation}
The first map is contractive by Lemma \ref{Lemma-contractivity}. Moreover, for any $x,y \in \cal{M}$, we have
$$
P_1(\varphi_x,\varphi_y)
\ov{\eqref{Def-product-bis-bis}}{=} (\varphi_y)_x.
$$
Since $\varphi$ is a trace, we have for any $z \in \cal{M}$
$$
(\varphi_y)_x(z)
\ov{\eqref{Translate-state}}{=} (\varphi_y)(x \circ z)
\ov{\eqref{Translate-state}}{=}  \varphi(y \circ (x \circ z))
\ov{\eqref{trace}}{=} \varphi((y \circ x) \circ z)
=\varphi((x \circ y) \circ z)
=\varphi_{x \circ y}(z).
$$
So $(\varphi_y)_x=\varphi_{x \circ y}$. The maps \eqref{Def-product-bis-bis} coincide on the product $i(\cal{M}) \times i(\cal{M})$. By bilinear interpolation (Theorem \ref{Th-bilinear-interpolation}), we obtain a contractive bilinear map $\cal{M} \times (\cal{M}_*,\cal{M})_{\frac{1}{p}}=\cal{M} \times \L^{p}(\cal{M}) \to \L^p(\cal{M})$, $(h,k) \mapsto h \circ k$.
\end{proof}

We finish with a version of H\"older's inequality in the \textit{tracial} case.

\begin{prop}
\label{prop-Holder}
Let $\cal{M}$ be a $\JBW^*$-algebra equipped with a normal faithful tracial state $\varphi$. Suppose that $1 \leq p,q,r \leq \infty$ such that $\frac{1}{r}=\frac{1}{p}+\frac{1}{q}$. The product $(x,y) \mapsto x \circ y$ on $\cal{M}$ induces a contractive map $\L^p(\cal{M}) \times \L^q(\cal{M}) \to \L^r(\cal{M})$, $(h,k) \mapsto h \circ k$, i.e.
\begin{equation}
\label{Holder-1}
\norm{h \circ k}_{\L^r(\cal{M})}
\leq \norm{h}_{\L^p(\cal{M})} \norm{k}_{\L^q(\cal{M})}, \quad h \in \L^p(\cal{M}),k \in \L^{q}(\cal{M}).
\end{equation}
\end{prop}

\begin{proof}
Consider the bilinear maps
$$
P_1 \co \cal{M} \times \L^r(\cal{M}) \to \L^r(\cal{M}), (h,k) \mapsto h \circ k
\quad \text{and} \quad
P_0 \co \L^r(\cal{M}) \times \cal{M} \to \L^r(\cal{M}), (h,k) \mapsto h \circ k.
$$
These maps coincide on the product $i(\cal{M}) \times i(\cal{M})$ with the map $(\varphi_x,\varphi_y) \mapsto \varphi_{x \circ y}$. Now, we choose $0 \leq \theta \leq 1$ such that $\frac{1}{q} = \frac{\theta}{r}$. We have $\frac{1}{p}=\frac{1-\theta}{r}$. By bilinear interpolation (Theorem \ref{Th-bilinear-interpolation}), we obtain a contractive bilinear map $(h,k) \mapsto h \circ k$ from $(\cal{M},\L^r(\cal{M}))_\theta \times (\L^r(\cal{M}),\cal{M})_\theta = \L^p(\cal{M}) \times \L^q(\cal{M})$ into $\L^r(\cal{M})$.
\end{proof}


Let $\cal{M}$ be a $\JBW^*$-algebra with product $(x,y) \mapsto x \circ y$. By essentially \cite[Theorem 5.1.29 p.~9]{CGRP18} \cite[Proposition 3.8.2 p.~91]{HOS84}, the selfadjoint part $A \ov{\mathrm{def}}{=} \cal{M}_\sa$ of the $\JBW^*$-algebra $\cal{M}$ is a $\JBW$-algebra. Suppose that $\tau$ is a normal finite faithful trace on $A$. We define an extension of $\tau$ on $\cal{M}$ by letting $\tau(x+\i y) \ov{\mathrm{def}}{=} \tau(x)+\i\tau(y)$ for any $x,y \in A$. By \cite[(5.1.2) p.~9]{CGRP18} \cite[Lemma 3.1]{HKPP20}, for any $x \in \cal{M}$ we have $x^* \circ x \geq 0$. Using the fact \cite[Proposition 3.4.1 p.~359]{CGRP14} that the closed subalgebra of a unital $\JB^*$-algebra generated by a selfadjoint element and 1 is a commutative $\mathrm{C}^*$-algebra, we can use functional calculus for defining
\begin{equation}
\label{norm-LpNA}
\norm{x}_{p}
\ov{\mathrm{def}}{=} \big(\tau \big[(x^* \circ x)^{\frac{p}{2}}\big]\big)^{\frac{1}{p}}, \quad x \in \cal{M}
\end{equation}
where $1 \leq p < \infty$, extending the norm 
\begin{equation}
\label{norm-Iochum}
\norm{x}_{\L^{p,I}(A,\tau)}
=\big(\tau\big[ |x|^{p}\big]\big)^{\frac{1}{p}}, \quad x \in A
\end{equation}
introduced in \cite[p.~139]{Ioc84} \cite[Definition 2]{Ioc86} on the $\JBW$-algebra $A$, where $|x| \ov{\mathrm{def}}{=} (x^2)^{\frac{1}{2}} \in A_+$. We conjecture that the definition \eqref{norm-LpNA} defines a norm and that the completion is isometric to the Banach space defined in Definition \ref{def-nonasso-Lp} (the case $p=2$ is Proposition \ref{L2-nonassociative}). In this case, this Banach space would contain isometrically the nonassociative $\L^p$-space $\L^{p,I}(A,\tau)$ of Iochum defined in \cite[Definition 4]{Ioc86}, as the completion of $A$ for the norm \eqref{norm-Iochum}. See also \cite{Abd83}. Note that $\L^{p,I}(A,\tau)$ is \textit{real} Banach space and that the space of real linear combination of pairwise orthogonal projections of $A$ is dense in the Banach space $\L^{p,I}(A,\tau)$. This last observation is also true for $p=\infty$ by \cite[Proposition 4.2.3 p.~99]{HOS84}.  


\begin{remark} \normalfont 
If $p=2$, we have the following link between the two norms \eqref{norm-LpNA} and \eqref{norm-Iochum}. If $a+\i b$ belongs to the $\JBW^*$-algebra $A+\i A=\cal{M}$ we have
\begin{align*}
\MoveEqLeft
\norm{a+\i b}_{2}         
\ov{\eqref{norm-LpNA}}{=} \big[\tau \big((a+\i b)^* \circ (a+\i b)\big)\big]^{\frac{1}{2}} 
=\big[\tau \big( a \circ a +\i a \circ  b-\i b \circ a +b \circ b  \big)\big]^{\frac{1}{2}}\\
&=\big[\tau( a \circ a +b \circ b)\big]^{\frac{1}{2}}=\big[\tau(a \circ a) +\tau(b \circ b)\big]^{\frac{1}{2}} \ov{\eqref{norm-Iochum}}{=} \big[\norm{a}_{\L^{2,I}(A,\tau)}^2 +\norm{b}_{\L^{2,I}(A,\tau)}^2\big]^{\frac{1}{2}}.
\end{align*}
\end{remark}




Now, we will connect our spaces to the nonassociative $\L^p$-spaces $\L^{p,I}(A,\tau)$ of Iochum. We need the real polar decomposition \cite[p.~13]{Top65} of an arbitrary element $x$ of $\JBW$-algebra $A$. We recall the short argument since it is not available in the literature and since \cite[p.~13]{Top65} only states the result for $\JW$-algebras. Let $\W(x)$ be the $\JBW$-algebra generated by $x$. By \cite[Lemma 4.1.11 p.~97]{HOS84} \cite[Proposition 2.11 p.~41]{AlS03}, it is associative and isomorphic to a monotone complete algebra $\mathrm{C}_\R(X)$ of real continuous functions on a compact Hausdorff space $X$. Recall that an ordered vector space $Z$ is monotone complete if for each bounded increasing net $(a_i)$ of $Z$ has a least upper bound $a$ in $Z$. By \cite[Proposition 1.7 p.~104]{Tak02}, this condition on the algebra $\mathrm{C}_\R(X)$ is equivalent to $X$ being extremely discontinuous. Let $f$ be a function of $\mathrm{C}_\R(X)$ and let
$$
\sgn(f)
\ov{\mathrm{def}}{=}\begin{cases}
1& \text{ if } f(x) \geq 0 \\
-1& \text{ if } f(x) < 0
\end{cases}.
$$ 
We have the equality $f=\sgn(f)|f|$ with $\sgn(f)^2=1$. So by a functional calculus argument,  there exists a symmetry $s \in A$ (i.e.~$s^2=1$) such that 
\begin{equation}
\label{real-polar}
x=s \circ |x|.
\end{equation}

Note that in the following statement, we have a canonical inclusion of $A$ in the space $\L^p(\cal{M},\tau)$ since $A$ is a real subspace of the $\JBW^*$-algebra $\cal{M}$.

\begin{thm}
\label{Th-Iochum}
Let $A$ be a $\JBW$-algebra equipped with a normal finite faithful trace $\tau$ and $\cal{M}$ be the $\JBW^*$-algebra associated to $A$. Suppose that $1 \leq p \leq \infty$. The canonical inclusion $A \subset\L^p(\cal{M},\tau)$ induces an isometry $\L^{p,I}(A,\tau) \xhookrightarrow{} \L^p(\cal{M},\tau)$.
\end{thm}

\begin{proof}
The case $p=\infty$ is obvious. In the sequel, we will use the notation $\L^p(\cal{M})$ for $\L^p(\cal{M},\tau)$. So, we can assume that $1 \leq p < \infty$. Let $\theta \ov{\mathrm{def}}{=} \frac{1}{p}$ and $\epsi>0$. Let $x \in A$ such that
\begin{equation}
\label{alphago}
\norm{x}_{\L^{p,I}(A,\tau)} 
\leq 1
\end{equation} 
with real polar decomposition $x = s \circ |x|$. We suppose that $x$ is a real linear combination of pairwise orthogonal projections $e_1,\ldots,e_n \in A$. We have $|x|=\sum_{k=1}^n \lambda_k e_k$ where $\lambda_1,\ldots,\lambda_n \in \R^+_*$. Recall that the closed subalgebra of a $\JB^*$-algebra generated by a selfadjoint element and 1 is a commutative $\mathrm{C}^*$-algebra, see \cite[Proposition 3.4.1 p.~359]{CGRP14}. For any $z \in \mathbb{C}$ with $0 \leq \Re z \leq 1$, we consider the element 
\begin{equation}
\label{f-epsi}
f_\epsi(z) 
\ov{\mathrm{def}}{=} \e^{\epsi(z^2-\theta^2)} s \circ |x|^{zp}
\end{equation}
of $\cal{M}$. Then $f_\epsi|_{S}$ is holomorphic and $f_\epsi$ is continuous and bounded on the closed strip $\ovl{S}$. Moreover, the functions $\mathbb{R} \to \L^{\infty}(\cal{M})$, $t \mapsto f_\epsi(\i t)$ and $\mathbb{R} \to \L^{1}(\cal{M})$, $t \mapsto f_\epsi(1+\i t)$ are continuous. Note that \eqref{def-JB} implies that $\norm{s}_{\L^{\infty}(\cal{M})}=\norm{s}_{A} \leq 1$.  Moreover, for any $t \in \R$, we have the estimates
\begin{align*}
\MoveEqLeft
\norm{f_\epsi(\i t)}_{\L^{\infty}(\cal{M})}        
\ov{\eqref{f-epsi}}{=}  \bnorm{\e^{\epsi(-t^2-\theta^2)} s \circ |x|^{\i t p}}_{\L^{\infty}(\cal{M})} 
=\e^{\epsi(-t^2-\theta^2)} \bnorm{ s \circ |x|^{\i t p}}_{\L^{\infty}(\cal{M})} \\
&\leq \e^{\epsi(-t^2-\theta^2)} \norm{s}_{\L^{\infty}(\cal{M})} \bnorm{|x|^{\i t p}}_{\L^{\infty}(\cal{M})}
\leq 1
\end{align*}
and using the isometry $\L^{1,I}(A,\tau)=A_*$ of \cite[Theorem 5]{Ioc86} in the last equality
\begin{align*}
\MoveEqLeft
\norm{f_\epsi(1+\i t)}_{\L^{1}(\cal{M})}        
=\bnorm{\e^{\epsi((1+\i t)^2-\theta^2)} s \circ |x|^{p(1+\i t)}}_{\L^{1}(\cal{M})} 
=|\e^{\epsi(1-t^2+2\i t-\theta^2)}| \bnorm{s \circ |x|^{p(1+\i t)}}_{\L^{1}(\cal{M})} \\
&\ov{\eqref{Holder-1}}{\leq} \e^{\epsi(1-t^2-\theta^2)} \bnorm{|x|^{p}}_{\L^{1}(\cal{M})} 
\ov{\eqref{norm-VU}}{=} \e^{\epsi(1-t^2-\theta^2)} \bnorm{|x|^{p}}_{\L^{1,I}(A,\tau)}
\ov{\eqref{norm-Iochum}\eqref{alphago}}{\leq} \e^{\epsi(1-\theta^2)}.
\end{align*}
These computations shows that $t \mapsto f_\epsi(\i t)$ and $t \mapsto f_\epsi(1+\i t)$ tend to 0 when $|t|$ goes to $\infty$. We deduce that
$$
\norm{f_\epsi}_{\cal{F}(\L^{1}(\cal{M}),\L^{\infty}(\cal{M}))}
\ov{\eqref{norm-funct-inter}}{\leq} \e^{\epsi(1-\theta^2)}.
$$
Since $f_\epsi(\theta) = s \circ |x|= x$, we infer that $x$ belongs to the interpolation space $(\L^{\infty}(\cal{M}),\L^{1}(\cal{M}))_\theta$ and that $\norm{x}_\theta \leq \e^{\epsi(1-\theta^2)}$. Letting $\epsi \to 0$, we obtain $\norm{x}_\theta \leq 1$. Then by homogeneity and by density in $\L^{p,I}(A,\tau)$, we obtain a contractive inclusion $\L^{p,I}(A,\tau) \subset (\L^{\infty}(\cal{M}),\L^{1}(\cal{M}))_\theta$.

Consider some $x \in A$ such that $\norm{x}_\theta < 1$. Let $y \in A$ with $\norm{y}_{\L^{p^*,I}(A,\tau)} \leq 1$ with real polar decomposition $y = s \circ |y|$. We will prove that
$$
|\tau(y \circ x)| 
\leq 1.
$$ 
By approximation, we can suppose that $|y|=\sum_{k=1}^n \lambda_k e_k$ is a linear combination of pairwise orthogonal projections $e_1,\ldots,e_n \in A$ where $\lambda_1,\ldots,\lambda_n \in \R^+_*$. 
Since $\norm{x}_\theta < 1$, there exists a function $f \in \mathscr{F}(\L^{\infty}(\cal{M}), \L^{1}(\cal{M}))$ such that $f(\theta) = x$ and $
\norm{f}_{\mathscr{F}(\L^{\infty}(\cal{M}),\L^{1}(\cal{M}))}
\leq 1$. For any $z \in \mathbb{C}$ with $0 \leq \Re z \leq 1$, we define
\begin{equation}
\label{def-g-h}
g(z) 
\ov{\mathrm{def}}{=} \e^{\epsi(z^2-\theta^2)} s \circ |y|^{p^*(1-z)}
\quad \text{and} \quad
h(z) 
\ov{\mathrm{def}}{=} \tau(g(z) \circ f(z)).
\end{equation}
Then clearly, the function $h \co \ovl{S} \to \mathbb{C}$ is continuous and bounded on the closed strip $\ovl{S}$ and analytic on the open strip $S$. Moreover, using the isometry $\L^{1,I}(A,\tau)=A_*$ of \cite[Theorem 5]{Ioc86} in the last equality, we have for any $t \in \R$
\begin{align*}
\MoveEqLeft
|h(\i t)| 
\ov{\eqref{def-g-h}}{=} |\tau(g(\i t) \circ f(\i t))|
\ov{\eqref{Translate-state}}{=} |\tau_{g(\i t)}(f(\i t))|
\leq \norm{\tau_{g(\i t)}} \norm{f(\i t)}_{\L^\infty(\cal{M})} \\
&= \norm{g(\i t)}_{\L^1(\cal{M})} \norm{f(\i t)}_{\L^\infty(\cal{M})}
= \e^{\epsi(-t^2-\theta^2)} \bnorm{s \circ |y|^{p^*(1-\i t)}}_{\L^1(\cal{M})} \\
&\ov{\eqref{Holder-1}}{\leq} \e^{\epsi(-t^2-\theta^2)} \bnorm{|y|^{p^*}}_{\L^1(\cal{M})} 
\ov{\eqref{norm-VU}}{=} \e^{\epsi(-t^2-\theta^2)} \bnorm{|y|^{p^*}}_{\L^{1,I}(A,\tau)}
\leq 1 
\end{align*}
and similarly
\begin{align*}
\MoveEqLeft
|h(1 + \i t)| 
\ov{\eqref{def-g-h}}{=} |\tau(g(1 + \i t) \circ f(1 + \i t))| 
\leq \norm{g(1 + \i t)}_{\L^{\infty}(\cal{M})} \norm{f(1 + \i t)}_{\L^{1}(\cal{M})} \\
&\ov{\eqref{def-g-h}}{=} \e^{\epsi(1-t^2-\theta^2)}\norm{s \circ |y|^{-p^*\i t}}_{\L^{\infty}(\cal{M})} \norm{f(1 + \i t)}_{\L^{1}(\cal{M})} 
\leq \e^{\epsi(1-\theta^2)}.
\end{align*}
Furthermore, we have the equality
$$
h(\theta)
\ov{\eqref{def-g-h}}{=} \tau(g(\theta) \circ f(\theta))
=\tau(yx).
$$
The Hadamard three-lines theorem (Proposition \ref{prop-three-lined}) 
gives the estimate $|h(\theta)| \leq \e^{\epsi(1-\theta^2)\theta}$, i.e.~$|\tau(yx)| \leq \e^{\epsi(1-\theta^2)\theta}$. Letting $\epsi \to 0$ we obtain $|\tau(yx)| \leq 1$. Taking the supremum over all $y$, we get the inequality $\norm{x}_{ \L^{p,I}(A,\tau)} \leq 1$.
\end{proof}

\section{Embeddings of nonassociative $\L^p$-spaces in noncommutative $\L^p$-spaces}
\label{Sec-nonassociative-Lp-spaces}

In this part, we will show that tracial nonassociative $\L^p$-spaces from $\JW^*$-factors arise as positively 1-complemented subspaces of noncommutative $\L^p$-spaces. We will use a proof by cases relying on the property of the associated $\JW$-algebra.  By the well-known decomposition result \cite[Theorem 6.4]{Sto66} of arbitrary $\JW$-algebras, it suffices to examine the case of selfadjoint parts of von Neumann algebras (which is obvious), the case of purely real $\JW$-factors and the one of $\JW$-algebras of type $\I_2$. We start with the purely real case.

\paragraph{Purely real $\JW$-algebras} A real von Neumann algebra \cite[p.~15]{ARU97} is a real unital $*$-subalgebra $R$ of $\B(H)$ which is weak operator closed satisfying $R \cap \i R=\{0\}$. Moreover, the (complex) von Neumann generated by $R$ is given by $R''=R+\i R$ and it is obvious that its selfadjoint part $A \ov{\mathrm{def}}{=} \{ x \in R : x^*=x \}$ is a reversible $\JW$-algebra with associated $\JW^*$-algebra $\cal{M}$. One can show \cite[pp.~21-22]{ARU97} or \cite[Lemma 3.2]{Sto68} that the map $\alpha \co A'' \to A''$, $z+\i y \mapsto z^*+\i y^*$ is a $*$-antiautomorphism of order 2 and it is easy to check that
\begin{equation}
\label{}
R=\big\{x \in R'' : \alpha(x)=x^* \big\}, 
\quad A
=\big\{x \in R'' : \alpha(x)=x=x^*\big\}
\end{equation}
and
$$
\cal{M}=\big\{x \in R'' : \alpha(x)=x\big\}.
$$
Actually, defining a real von Neumann algebra is equivalent to defining an involutory $*$-anti-automorphism of a von Neumann algebra. 
Finally, the map $P_\can \ov{\mathrm{def}}{=} \frac{\Id+\alpha}{2} \co A'' \to A''$ is clearly a positive contractive normal unital projection and we said that it is the canonical projection of $A''$ onto $\cal{M}$.

Given a $\JW$-algebra $A$, we denote by $\kR(A)$ the closure for the weak* topology of the real algebra generated by $A$ in $\B(H)$ (note that this algebra is closed under adjoints and  sometimes denoted $\ovl{\kR(A)}$ in the literature). If $A$ is a reversible then 
\begin{equation}
\label{}
A=\kR(A)_\sa
\quad \text{and} \quad 
A''=\kR(A)+\i \kR(A)
\end{equation}
by \cite[Lemma 4.25 p.~113]{AlS03}, \cite[Theorem 1.1.5 p.~17]{ARU97} and \cite[Theorem 2.4]{Sto68}. A reversible $\JW$-algebra $A$ is said to be purely real \cite[p.~15]{ARU97} if $\kR(A)$ is a real von Neumann algebra, i.e.~$\kR(A) \cap \i\kR(A) =\{0\}$. 


We will use the following property \cite[Proposition 1.5.1 p.~49]{ARU97}.

\begin{prop}
\label{prop-Ayu151}
A purely real $\JW$-factor $A$ is not isomorphic to the selfadjoint part of a von Neumann algebra if and only if the von Neumann algebra $A''$ is a factor.
\end{prop}

\begin{remark} \normalfont
The property of being purely real is not an invariant under isomorphisms. See \cite[p.~1427]{Ayu87}.
\end{remark}

The following is a particular case of \cite[Corollary 1.2.10 p.~35]{ARU97}. 

\begin{prop}
\label{prop-Ayu1210}
Let $A$ be a purely real $\JW$-algebra equipped with a normal trace $\tau$. Then the trace $\tau$ can be extended to a normal trace $\tau_1$ on the von Neumann algebra $A''$. If $\tau$ is faithful (respectively finite or semifinite) then $\tau_1$ is also faithful (respectively finite or semifinite).
\end{prop}

In the following statement, observe that the normal finite faithful trace $\tau$ extends to a normal finite faithful trace $\tau$ on the von Neumann algebra $A''$ by Proposition \ref{prop-Ayu1210}.

\begin{prop}
\label{Prop-complementation-1}
Let $\cal{M}$ be a $\JW^*$-algebra whose associated $\JW$-algebra is $A$. Let $\tau$ be a normal finite faithful trace on $A$. Suppose that $A$ is a purely real $\JW$-factor which is not isomorphic to the selfadjoint part of a von Neumann algebra. Then $P_\can \co A'' \to A''$ is a selfadjoint normal unital positive projection whose range is $\cal{M}$. This maps extends to a positive contractive projection $P_{\can,p} \co \L^p(A'') \to \L^p(A'')$ whose range is $\L^p(\cal{M},\tau)$ for any $1 \leq p \leq \infty$.
\end{prop}

\begin{proof}
By the previous discussion, there exists an involutive $*$-anti-automorphism $\alpha \co A'' \to A''$ such that $\cal{M}=\{ x \in A'' : \alpha(x)=x \}$. By Proposition \ref{prop-Ayu151}, the von Neumann algebra $A''$ is a factor. Since $\tau$ is finite, the von Neumann algebra $A''$ is finite. By \cite[Theorem 8.2.8 p.~517]{KaR97b}, we deduce that the normalized finite trace of $A''$ is unique. Since $\tau \circ \alpha$ is also a normalized finite trace on the von Neumann algebra $A''$, we conclude that $\alpha$ is trace preserving, that is $\tau \circ \alpha=\tau$. For any $x,y \in A''$, we have
$$
\tau(\alpha(x)y)
=\tau(\alpha(\alpha(x)y))
=\tau(\alpha(y)\alpha^2(x))
=\tau(\alpha(y) x)
=\tau(x\alpha(y)).
$$
So $\alpha$ is selfadjoint in the sense of \cite[p.~43]{JMX06}. Finally, we consider the canonical projection $P_\can \co A'' \to A''$, $x \mapsto \frac{1}{2}(x+\alpha(x))$. It is clear that $P_\can$ is a selfadjoint normal unital positive projection which induces a positive contractive projection $P_{\can,p} \co \L^p(A'') \to \L^p(A'')$ 

By Proposition \ref{Prop-Fabian}, the range of the preadjoint map $P_{\can,*} \co A''_* \to A''_*$ is isometric to the predual $\cal{M}_*$ of the range $\cal{M}$ of $P_{\can}$. Since $P_{\can}$ is selfadjoint, note that by \cite[p.~43]{JMX06} the preadjoint $P_{\can,*} \co A''_* \to A''_*$ identifies to the $\L^1$-extension map $P_{\can,1} \co \L^1(A'') \to \L^1(A'')$. So the range of $P_{\can,1}$ is also $\cal{M}_*$. Using Lemma \ref{Lemma-interpolation} in the second equality with the interpolation couple $(A'',A''_*)$ and $C=\cal{M}_*$, we obtain isometrically
$
\Ran P_{\can,p}
		=(\cal{M},\cal{M}_*)_{\frac{1}{p}} 
		\ov{\eqref{Def-Lp-state-JBW}}{=} \L^p(\cal{M},\tau)
$ 
where we use some identifications in the first equality which are left to the reader (use the projection for decomposing $\A''$ in a sum of two subspaces).
%
%
%
%
%
%
%
%
\end{proof}


\begin{example} \normalfont
Recall that by \cite[Theorem 1.7.2 p.~67]{ARU97}, there exist exactly two non-isomorphic injective $\JW$-factors of type $\II_1$ with separable predual. The first factor $\cal{R}_\sa$ is the selfadjoint part of the unique injective factor $\cal{R}$ of type $\II_1$ with separable predual. We can describe the the second as the set $\cal{A} \ov{\mathrm{def}}{=} \{x \in \cal{R}_\sa : \alpha(x)=x \}$ of fixed selfadjoint points of $\cal{R}$ under an involutive $*$-antiautomorphism $\alpha \co \cal{R} \to \cal{R}$ of the von Neumann algebra $\mathcal{R}$. Although there exists a lot of involutions acting on $\cal{R}$, the associated sets of fixed points are known to be isomorphic by \cite[Corollary 1.5.11 p.~57]{ARU97}, since all involutions of $\cal{R}$ are conjugated by \cite[Corollary 2.10]{Sto80b} (see also \cite{Gio83a} and \cite{Gio83b}). Note that by \cite[Theorem 1.5.2 p.~49]{ARU97}, the $\JW$-algebra $\cal{A}$ is not isomorphic to the selfadjoint part of a von Neumann algebra. By Proposition \ref{Prop-complementation-1}, we observe that $\L^p(\cal{A})$ is positively 1-complemented in the noncommutative $\L^p$-space $\L^p(\cal{R})$ by the positive contractive projection $P_p \co \L^p(\cal{R}) \to \L^p(\cal{R})$, $x \mapsto \frac{1}{2}(x+\alpha_p(x))$.
\end{example}

\paragraph{Type $\I_2$ $\JW$-algebras}

We investigate the last case. Our analysis relies on the structure of $\JW$-algebra of type $\I_2$.

\begin{prop}
\label{prop-spin}
Let $\cal{M}$ be a $\JW^*$-algebra with separable predual such that the associated $\JW$-algebra $A \ov{\mathrm{def}}{=}\cal{M}_\sa$ is of type $\I_2$. Let $\tau$ be a normal finite faithful trace on $A$. Suppose that $1 \leq p < \infty$. Then $\L^p(\cal{M},\tau)$ is isometric to a positively 1-complemented subspace of a noncommutative $\L^p$-space associated with a finite von Neumann algebra $\cal{N}$ whose contractive projection is induced by a selfadjoint positive normal projection $P \co \cal{N} \to \cal{N}$.
\end{prop}

\begin{proof}
By complexification, we have a normal finite faithful trace $\tau$ on $\cal{M}$ satisfying \eqref{trace}. By Example \ref{ex-type-II-JW}, there exist an index set $I$, a family $(\Omega_i)_{i \in I}$ of second countable locally compact spaces, a family of Radon measures $(\mu_i)_{i \in I}$ on the spaces $\Omega_i$ and a family $(S_i)_{i \in I}$ of spin factors as in Example \ref{spin-factor}, with associated real Hilbert space $\cal{H}_i$  each of dimension at most countable and strictly greater than 1, giving an isomorphisms 
$$
A
=\oplus_{i \in I} \L^\infty_\R(\Omega_i,S_i) 
\quad \text{and} \quad
\cal{M}
=\oplus_{i \in I} \L^\infty(\Omega_i,\cal{S}_i)
$$ 
where $\cal{S}_{i}$ is the $\JW^*$-algebra (<<complex spin factor>>) associated to $S_i$. The trace $\tau$ induces a finite trace $\int_{\Omega_i} \cdot \d \nu_i \ot \tau_i$ on each $\JW$-algebra $\L^\infty_\R(\Omega_i,S_i)$ where $\tau_i$ is the canonical normalized trace on $S_i$ as in Example \ref{ex-trace-spin}. If $\cal{A}$ is the CAR algebra, we can suppose by \cite[Theorem 6.2.2 p.~141]{HOS84} that the $\mathrm{C}^*$-algebra generated by the spin factor $S_i$ is $*$-isomorphic to
\begin{equation*}
\begin{cases}
\M_{2^{n-1}}(\mathbb{C}) \oplus \M_{2^{n-1}}(\mathbb{C}) & \text{ if } \dim \cal{H}_i=2n-1 \\
\M_{2^n}(\mathbb{C}) & \text{ if } \dim \cal{H}_i=2n \\
\cal{A} &\text{ if } \dim \cal{H}_i= \infty
\end{cases}.
\end{equation*}
For any $i \in I$, we consider the following von Neumann algebra
\begin{equation*}
\cal{N}_i
\ov{\mathrm{def}}{=} 
\begin{cases}
\M_{2^{n-1}}(\mathbb{C}) \oplus \M_{2^{n-1}}(\mathbb{C}) & \text{ if } \dim \cal{H}_i=2n-1 \\
\M_{2^n}(\mathbb{C}) & \text{ if } \dim \cal{H}_i=2n \\
\cal{R} &\text{ if } \dim \cal{H}_i= \infty
\end{cases}
\end{equation*} 
where $\cal{R}$ is the unique injective factor of type $\II_1$ with separable predual. Here, we use the realization of $\cal{R}$ given by the weak closure of the range of the GNS representation of the CAR algebra $\cal{A}$ for the tracial state $\rho^{\ot \infty}$ where $\rho$ is the state defined on $\M_{2}(\mathbb{C})$ by $\rho=\frac{1}{2} \tr$, see \cite{KaR97b} for more information on this construction and this result. We equip $\cal{R}$ with the restriction of the vector state given by this GNS representation. It is the unique normal tracial state of $\cal{R}$ and its restriction on $\cal{A}$ is equal to $\rho$. Finally, we equip the factor $\M_{2^n}(\mathbb{C})$ with its normalized trace and $\M_{2^{n-1}}(\mathbb{C}) \oplus \M_{2^{n-1}}(\mathbb{C})$ with the normalized trace obtained from $\tr \oplus \tr$.



Note that the trace of each operator $s_k$ of Example \ref{spin-factor} is equal to 0. So we can suppose that each $\cal{S}_i$ is embedded as a $\JW^*$-subalgebra of $\cal{N}_i$ with the canonical trace on $\cal{S}_i$ (the complexification of the trace defined in Example \ref{ex-trace-spin}) extended to a normal faithful trace $\tau_i$ of $\cal{N}_i$. Hence we can see the $\JW^*$-algebra $\L^\infty(\Omega_i,\cal{S}_i)$ as a $\JW^*$-subalgebra of the $\JW^*$-algebra underlying to the von Neumann algebra $\L^\infty(\Omega_i,\cal{N}_i)$. 

We equip the von Neumann algebra $\L^\infty(\Omega_i,\cal{N}_i)$ with the normalized trace $\int_{\Omega_i} \cdot \d \nu_i \ot \tau_i$. 
By Proposition \ref{prop-Jordan-conditional-expectation-existence-bis}, there exists a trace preserving normal faithful Jordan conditional expectation $Q_i \co \L^\infty(\Omega_i,\cal{N}_i) \to \L^\infty(\Omega_i,\cal{N}_i)$ onto the $\JW^*$-subalgebra $\L^\infty(\Omega_i,\cal{S}_i)$. By Proposition \ref{Prop-selfadjoint}, this map is selfadjoint normal positive and unital. By \cite[p.~43]{JMX06}, this map induces a positive contractive projection $Q_{i,p} \co \L^p(\Omega_i,\cal{N}_i) \to \L^p(\Omega_i,\cal{N}_i)$.


%

By Proposition \ref{Prop-Fabian}, the range of the preadjoint map $Q_{i,*} \co \L^1(\Omega_i,\cal{N}_i) \to \L^1(\Omega_i,\cal{N}_i)$ is isometric to the predual $(\L^\infty(\Omega_i,\cal{S}_i))_*=\L^1(\Omega_i,\L^1(\cal{S}_{i}))$ of the range of $Q_i$. Since $Q_i$ is selfadjoint, note that by \cite[p.~43]{JMX06} the preadjoint $Q_{i,*} \co \L^1(\Omega_i,\cal{N}_i) \to \L^1(\Omega_i,\cal{N}_i)$ identifies to the $\L^1$-extension map $Q_{i,1} \co \L^1(\Omega_i,\cal{N}_i) \to \L^1(\Omega_i,\mathcal{N}_i)$. So the range of $Q_{i,1}$ is also $\L^1(\Omega_i,\L^1(\cal{S}_{i}))$. Now, we use Lemma \ref{Lemma-interpolation} with the interpolation couple $(\L^\infty(\Omega_i,\cal{N}_i),\L^1(\Omega_i,\L^1(\mathcal{N}_i)))$ and $C=\L^1(\Omega_i,\L^1(\cal{S}_{i}))$. We obtain isometrically
\begin{align*}
\MoveEqLeft
\big(\L^\infty(\Omega_i,\cal{N}_i) \cap \L^1(\Omega_i,\L^1(\cal{S}_{i})),\L^1(\Omega_i,\L^1(\cal{N}_i)) \cap \L^1(\Omega_i,\L^1(\cal{S}_{i}))\big)_{\frac{1}{p}} \\
		&=\big(\L^\infty(\Omega_i,\cal{S}_i),\L^1(\Omega_i,\L^1(\cal{S}_i))\big)_{\frac{1}{p}} 
		=\L^p(\Omega_i,\L^p(\cal{S}_i)).
\end{align*}
We infer that the space $\Ran Q_p$ is isometric to the nonassociative $\L^p$-space $\L^p(\Omega_i,\L^p(\cal{S}_i))$. We deduce that this space is a positively 1-complemented subspace of the noncommutative $\L^p$-space $\L^p(\Omega_i,\L^p(\cal{N}_i))$. We conclude with a direct sum argument.
\end{proof}


The results of this section allows us to state the following theorem.

\begin{thm}
\label{th-embedding-intro}
Let $\cal{M}$ be a $\JW^*$-factor with separable predual whose associated $\JW$-algebra $A$ is equipped with a normal finite faithful trace $\tau$. Suppose that $1 \leq p < \infty$. Then $\L^p(\cal{M},\tau)$ is isometric to a positively 1-complemented subspace of a noncommutative $\L^p$-space associated with a finite von Neumann algebra.
\end{thm}

\begin{remark} \normalfont
If $\cal{M}$ is equipped with a normal \textit{semifinite} faithful trace $\tau$ instead of a normal finite faithful trace, we probably obtain the same result but we did not check all the details. Indeed, similarly to the setting of von Neumann algebras, if a $\JW$-factor $\cal{M}$ admits a normal semifinite faithful trace, then probably it is unique up to scaling. 
In this case, with the notations of the proof of Proposition \ref{Prop-complementation-1}, we have $\tau \circ \alpha=\lambda \tau$ for some scalar $\lambda >0$. Hence
$$
\tau
=\tau \circ \alpha \circ \alpha
=\lambda \tau \circ \alpha
=\lambda^2 \tau.
$$
So $\tau \circ \alpha= \tau$. Consequently slight modifications probably allows us to prove a generalization of Theorem \ref{th-embedding-intro}. 
\end{remark}

\begin{remark} \normalfont
We have not attempted to address the case where the predual is not separable.
\end{remark}
\section{Open questions}
\label{sec-open-questions}

Let $\cal{M}$ be a von Neumann algebra. Recall that the noncommutative $\L^p$-space $\L^p(\cal{M},\varphi)$ is independent of the normal semifinite faithful weight $\varphi$ up to an isometric isomorphism, see \cite[p.~59]{Ter81} and \cite[Theorem 5.1]{Ray03}. So the next question is natural.
 
\begin{quest} 
\label{Open-quest}
Suppose that $1 < p < \infty$. Let $\cal{M}$ be a $\JBW^*$-algebra equipped with a normal faithful state $\varphi$. Is it true that the Banach space $\L^p(\cal{M},\varphi)$ is independent from $\varphi$ ?
\end{quest}

In the case where the normal faithful state $\varphi$ is a trace, we have H\"older's inequality \eqref{Holder-1}. At the present moment, we have no similar result if the state $\varphi$ is not a trace.

\begin{quest} 
\label{Open-quest}
Suppose that $1 < p < \infty$. Let $\cal{M}$ be a $\JBW^*$-algebra equipped with a normal faithful state $\varphi$. Does there exist a generalization of H\"older's inequality ?
\end{quest}

In Section \ref{Sec-nonassociative-Lp-spaces}, we prove some particular cases of the following natural conjecture. 

\begin{conj}
\label{th-embedding}
Let $\cal{M}$ be a $\JW^*$-algebra equipped with a normal semifinite faithful trace $\tau$. Suppose that $1 \leq p < \infty$. Then the nonassociative $\L^p$-space $\L^p(\cal{M},\tau)$ is isometric to a positively 1-complemented subspace of a noncommutative $\L^p$-space associated with a semifinite von Neumann algebra.
\end{conj}

We think that it should be possible to prove this conjecture using direct integral theory for a $\JW^*$-algebra $\cal{M}$ with separable predual. Indeed, it is stated in \cite[Proposition VII.3.1 p.~184]{Ioc84} (see also \cite[Proposition 7.3]{BeI78}, \cite{BeI79} and \cite[p.~162]{BeI80}) that we can write the associated $\JW$-algebra $A$ as a direct integral of $\JBW$-factors, i.e.~there exists a standard Borel space $\Omega$, a finite positive Borel measure $\mu$ on $\Omega$, a measurable field $\omega \mapsto A_\omega$ which is $\mu$-integrable providing an isomorphism
$$
A
=\int_\Omega^\oplus A_\omega \d \mu(\omega)
$$ 
such that each $\JBW$-algebra $A_\omega$ is a factor for $\mu$-almost every $\omega$ with $\Zc(A)=\L^\infty_\R(\Omega,\mu)$. Unfortunately, the sense of this direct integral is unclear. No definition seems to be provided in the previous references.

Note the next question. The answer should be fairly easy. The particular case $p=2$ is true by the first part of Proposition \ref{L2-nonassociative}.

\begin{quest}
Is it true that the formula \eqref{norm-LpNA} defines a norm and that the completion is isometric to the Banach space $\L^{p}(\cal{M},\tau)$ ?
\end{quest}

We finish with the following question.

\begin{quest}
\label{quest-1}
Suppose that $1 \leq p < \infty$ with $p \not=2$. Is it true that each positively contractively complemented subspace of a tracial nonassociative $\L^p$-space $\L^p(\cal{M},\tau)$ is isometric to a (not necessarily tracial) nonassociative $\L^p$-space ? 
\end{quest}

Observe that the positivity assumption is essential here. Indeed, in the case where $p=1$ and where $\cal{M}$ is the von Neumann algebra $\B(H)$ of bounded operators on a complex Hilbert space $H$, it is elementary to find contractively complemented subspaces which are isometric to Hilbert spaces. 

Finally, note that the predual of a $\JBW^*$-algebra is of course the predual of a $\JBW^*$-triple and that the range of a contractive projection on the predual of a $\JBW^*$-triple is isometric to the predual of a $\JBW^*$-triple.
\section{Appendix: existence and uniqueness of Jordan conditional expectations}
\label{Existence}


Let $B$ be a (unital) $\JBW$-subalgebra of a $\JBW$-algebra $A$. A map $Q \co A \to A$ is called a Jordan conditional expectation on $B$ if it is a unital positive map of range $B$ which is $B$-modular, that is 
\begin{equation}
\label{Def-cond-exp-JB}
Q(x \circ Q(y))
=Q(x) \circ Q(y), \quad x,y \in A.
\end{equation}




Now, we give a complete proof of the existence and the uniqueness of a Jordan conditional expectation in the case of a normal faithful \textit{tracial} state.

\begin{prop}
\label{Prop-Jordan-conditional-expectation-existence}
Let $A$ be a $\JBW$-algebra and let $B$ be a $\JBW$-subalgebra of $A$. Let $\tau$ be a normal finite faithful trace on $A$. For each element $x$ of $A$, there exists a unique element $Q(x)$ of $B$ such that for any $y \in B$ we have
\begin{equation}
\label{Def-J-cond-exp}
\tau(Q(x) \circ y)
=\tau(x \circ y).
\end{equation}
The mapping $Q \co A \to A$ is a trace preserving normal faithful Jordan conditional expectation onto $B$ and such a map is unique.
\end{prop}

\begin{proof}
If $x \in A_+$, we let $\rho(y) \ov{\mathrm{def}}{=} \tau(x \circ y)$ for any $y \in B$. This form is positive and normal by \cite[Corollary 5.20 p.~149]{AlS03} and \cite[Proposition 2.4 p.~38]{AlS03}.

We always consider $x \in A_+$. By \cite[(1.11)]{AlS03}, the element $\norm{x}1-x$ is positive. Replacing $x$ by $\norm{x}1-x$, it follows that the map $y \mapsto \tau((\norm{x}-x) \circ y)=\norm{x}\tau(y)-\rho(y)$ is a positive linear functional on $B$. Consequently $0 \leq \rho \leq \norm{x} \tau$. From the Radon-Nikodym type theorem \cite[Theorem 3.3 and Proposition 3.5]{AyA85} (see also \cite[Lemma 5.27 p.~156]{AlS03} and \cite{Kin83}), there exists $Q(x)$ in $B_+$ such that 
$$
\rho(y)
=\tau(Q(x) \circ y)
$$ 
for each $y$ in $B$. This proves \eqref{Def-J-cond-exp} when $x$ is a positive element.

An arbitrary $x \in A$ can be expressed by \cite[Proposition 1.28 p.~15]{AlS03} as a linear combination of two positive elements of $A$. From what we have proved, there exists an element $x_0$ of $B$ such that
\begin{equation}
\label{Equa-inter-1244}
\tau(x_0 \circ y)
=\tau(x \circ y), \quad y \in B.
\end{equation}
Suppose that another element $x_1$ of $B$ has this property. For any $y$ in $B$, we have $\tau((x_1-x_0) \circ y)=\tau(x_1\circ y)-\tau(x_0 \circ y)\ov{\eqref{Equa-inter-1244}}{=} \tau(x \circ y)-\tau(x \circ y)=0$. In particular, $\tau((x_1-x_0) \circ (x_1-x_0))=0$. Since $\tau$ is faithful and since $(x_1-x_0) \circ (x_1-x_0)$ is positive by definition \cite[p.~82]{HOS84} the equality $x_1-x_0=0$. So $x_1=x_0$. Hence there exists a unique element $x_0$ of $B$ that satisfies \eqref{Equa-inter-1244}. We let $Q(x) \ov{\mathrm{def}}{=} x_0$. So we have \eqref{Def-J-cond-exp}.

If $x_1,x_2 \in A$, $y \in B$ and $\lambda \in \R$, we have
\begin{align*}
\MoveEqLeft
\tau(Q(\lambda x_1+ x_2) \circ y)
\ov{\eqref{Def-J-cond-exp}}{=} \tau((\lambda x_1+ x_2) \circ y)
=\lambda \tau(x_1 \circ y)+\tau(x_2 \circ y)  \\          
&\ov{\eqref{Def-J-cond-exp}}{=} \lambda\tau(Q(x_1) \circ y)+\tau(Q(x_2) \circ y)
=\tau((\lambda Q(x_1)+ Q(x_2)) \circ y).
\end{align*} 
Furthermore, if $z \in B$, $x,y \in A$ and observing that $Q(y) \circ z$ belongs to $B$, we see that
\begin{align*}
\MoveEqLeft
\tau(Q(x \circ Q(y)) \circ z)            
\ov{\eqref{Def-J-cond-exp}}{=} \tau((x \circ Q(y)) \circ z)
\ov{\eqref{Def-trace}}{=} \tau(x \circ (Q(y) \circ z)) \\
&\ov{\eqref{Def-J-cond-exp}}{=} \tau(Q(x) \circ (Q(y) \circ z))
\ov{\eqref{Def-trace}}{=} \tau((Q(x) \circ Q(y)) \circ z).
\end{align*} 
Moreover, for any $y \in B$ we have $
\tau(Q(1) \circ y)
=\tau(1 \circ y)$. From these equations and the uniqueness, we obtain that
$$
Q(\lambda x_1+ x_2)
=\lambda Q(x_1)+Q(x_2), \quad 
Q(x \circ Q(y))=Q(x) \circ Q(y)
\quad \text{and} \quad 
Q(1)=1.
$$ 
Furthermore, if $x \in A_+$, we deduce from the first part of the proof that $Q(x)$ is an element of $B_+$. Hence $Q$ is a Jordan conditional expectation onto $B$. Replacing $y$ by $1$ in \eqref{Def-J-cond-exp}, we infer that $Q$ is trace preserving. Suppose that $x \in A_+$ and $Q(x)=0$. Since $\tau$ is faithful and $\tau(x)=\tau(Q(x))=0$, we see that $x=0$. Hence $Q$ is faithful.

In order to prove that $Q$ is weak*-to-weak* continuous, we have to show that $\varphi \circ Q \in A_*$ for any $\varphi \in A_*$. By \cite[Proposition 4.5.3 p.~112]{HOS84}, the predual $A_*$ is the linear span of its positive elements. So we can assume that $\varphi$ is a positive normal linear form on $A$. Then $\varphi \circ Q$ is positive, and we have to prove that $\varphi \circ Q$ is normal. Suppose that $(x_i)$ is a bounded increasing net of elements of $A$, with least upper bound $x$ in $A$. Since $Q$ is a positive map, the net $(Q(x_i))$ in $B$ is bounded increasing, has $Q(x)$ as an upper bound in $B$, and consequently has a least upper bound $y$ in $B$ satisfying $y\leq Q(x)$. Since $\tau$ is normal, we have
$$
\tau(y)
=\lim_i \tau(Q(x_i))
=\lim_i \tau(x_i)
=\tau(x)
=\tau(Q(x)).
$$ 
So $\tau(Q(x)-y)=0$. By the faithfulness of $\tau$ and since $Q(x)-y \geq 0$, we conclude that $Q(x)=y$. Since $\varphi$ is normal, we obtain that 
$$
\lim_i \varphi(Q(x_i))
=\varphi(y)
=\varphi(Q(x)).
$$ 
Thus $\varphi \circ Q$ is normal by definition \cite[p.~94]{HOS84}. We conclude that $Q$ is weak*-to-weak* continuous.

Suppose that $Q_0 \co A \to A$ is another trace preserving normal faithful Jordan conditional expectation onto $B$. For any $x \in A$ and any $y \in B$, we have
\begin{align*}
\MoveEqLeft
\tau\big[ (Q-Q_0)(x) \circ y \big]
\ov{\eqref{Def-cond-exp-JB}}{=} \tau\big[ (Q-Q_0)(x \circ y) \big] 
=\tau\big[ Q(x \circ y) \big] -\tau\big[Q_0(x \circ y) \big] \\
&=\tau(x \circ y)-\tau(x \circ y) 
=0.
\end{align*}
Using the description \cite[Theorem 5]{Ioc86} of the predual of $B$, we conclude that $Q(x)=Q_0(x)$.
\end{proof}
%
%
%
%
%
%

\textbf{Acknowledgment}.
The author acknowledges support by the grant ANR-18-CE40-0021 (project HASCON) of the French National Research Agency ANR. The author would like to thank Miguel Cabrera Garc{\'i}a and {\'A}ngel Rodr{\'\i}guez Palacios for providing some information on $\JBW^*$-algebras and $\JBW^*$-triples. We extend our thanks to the referees for their insightful comments and valuable suggestions. 


{\footnotesize


\noindent C\'edric Arhancet\\ 
\noindent 6 rue Didier Daurat, 81000 Albi, France\\
URL: \href{http://sites.google.com/site/cedricarhancet}{https://sites.google.com/site/cedricarhancet}\\
cedric.arhancet@protonmail.com\\

}

\end{document}